\documentclass[11pt,reqno]{amsart}

\usepackage{amsmath}
\usepackage{amssymb}
\usepackage{amsthm}

\newtheorem{theorem}{Theorem}[section]
\newtheorem{prop}[theorem]{Proposition}
\newtheorem{lem}[theorem]{Lemma}
\newtheorem*{cor}{Corollary}
\theoremstyle{definition}
\newtheorem{defn}[theorem]{Definition}

\theoremstyle{remark}
\newtheorem*{rem}{Remark}
\numberwithin{equation}{section}

\begin{document}

\title[Separability idempotents in $C^*$-algebras]
{Separability idempotents in $C^*$-algebras}

\author{Byung-Jay Kahng}
\date{}
\address{Department of Mathematics and Statistics\\ Canisius College\\
Buffalo, NY 14208, USA}
\email{kahngb@canisius.edu}

\author{Alfons Van Daele}
\date{}
\address{Department of Mathematics\\ University of Leuven\\ Celestijnenlaan 200B\\ 
B-3001 Heverlee, BELGIUM}
\email{Alfons.VanDaele@wis.kuleuven.be}


\begin{abstract}
In this paper, we study the notion of a {\em separability idempotent\/} in the $C^*$-algebra 
framework.  This is analogous to the notion in the purely algebraic setting, typically considered 
in the case of (finite-dimensional) algebras with identity, then later also considered in the multiplier 
algebra framework by the second-named author.  The current work was motivated by the appearance 
of such objects in the authors' ongoing work on locally compact quantum groupoids.
\end{abstract}
\maketitle

{\sc Introduction.}

Consider a {\em groupoid\/} $G$ over the set of units $G^{(0)}$, together with the maps $s:G\to
G^{(0)}$ and $t:G\to G^{(0)}$ (the ``source map'' and the ``target map'').  This means 
that there is a set of ``composable pairs'' $G^{(2)}=\bigl\{(p,q)\in G\times G: s(p)=t(q)\bigr\}$, 
on which the product $pq$ in $G$ is defined.  This product is assumed to be associative, 
in an appropriate sense.  The set of units, $G^{(0)}$, may be naturally regarded as a subset 
of $G$.  There exists also the inverse map $p\mapsto p^{-1}$ (so that $(p^{-1})^{-1}=p$), 
for which we have $s(p^{-1})=t(p)$, $t(p^{-1})=s(p)$, and satisfying some natural conditions. 
For a more detailed discussion on the definition and the basic theory of groupoids, refer to 
\cite{Brbook},  \cite{Higbook}.  The groupoid notion can be further extended to incorporate 
locally compact topology, which is the notion of a {\em locally compact groupoid\/}.  For this, 
refer to \cite{Renbook}, \cite{Patbook}.

Suppose $G$ is a groupoid and consider $A=K(G)$, the set of all complex-valued functions 
on $G$ having finite support.  For the time being, let us disregard any topology on $G$. 
Under the pointwise multiplication, $A$ becomes a commutative algebra.

In particular, if $G$ is a finite groupoid, the algebra $A$ becomes unital ($1\in A$).  In that case, 
it is known that $A$ can be given a structure of a {\em weak Hopf algebra\/}, together with the 
map $\Delta$ from A into $A\otimes A$ (algebraic tensor product), defined by
$$
(\Delta f)(p,q):=\left\{\begin{matrix}f(pq) & {\text { if $s(p)=t(q)$}} \\ 0 & {\text { otherwise}}
\end{matrix}\right..
$$
The $\Delta$ map is referred to as the ``comultiplication'' (or ``coproduct'').  A weak Hopf algebra 
is in general noncommutative, but includes the above example as a fundamental case.  For more 
discussion on weak Hopf algebras, refer to \cite{BNSwha1}, \cite{BSwha2}. 
In a sense, a weak Hopf algebra is a finite quantum groupoid (see also \cite{Yagroupoid}, \cite{NVfqg}, 
\cite{Valfqg}).  Recently, at the purely algebraic level, this notion has been further generalized to include 
the case of non-finite groupoids (so the algebra $A$ is non-unital), which is the notion of a {\em weak 
multiplier Hopf algebra\/} developed by one of us (Van~Daele) and Wang.  Refer to \cite{VDWangwha0}, 
\cite{VDWangwha1}.

If $(A,\Delta)$ is a weak multiplier Hopf algebra, there exists a certain canonical idempotent 
element $E$, playing the role of ``$\Delta(1)$''.  Among the main properties of $E$ is the fact 
that it is a {\em separability idempotent\/}.  Namely, there exist algebras $B$ and $C$ such that 
$E$ is an idempotent element contained in the multiplier algebra $M(B\otimes C)$, and it satisfies 
some number of conditions.  The case that is relevant to us is when $A$ is a ${}^*$-algebra 
and is ``regular'' (see Section~4 of \cite{VDWangwha1}).  As a consequence, it turns 
out that there exist bijective anti-automorphisms $S_B:B\to C$ and $S_C:C\to B$ such that 
$$
E(b\otimes 1)=E\bigl(1\otimes S_B(b)\bigr),\qquad (1\otimes c)E=\bigl(S_C(c)\otimes1\bigr)E,
$$
for $b\in B$, $c\in C$.  

In the special (commutative) case of the weak multiplier Hopf algebra $A=K(G)$ for a groupoid 
$G$, let $B$ be the subalgebra of $M(A)$ given by the pull-back of the algebra $K(G^{(0)})$ 
via the source map $s:G\to G^{(0)}$, and $C\subseteq M(A)$ the pull-back of the algebra 
$K(G^{(0)})$ via the target map $t:G\to G^{(0)}$.  Then we will have $E\in M(B\otimes C)$. 
The maps $S_B$, $S_C$ would come from the antipode map of the weak multiplier Hopf algebra.

One of the main reasons for studying separability idempotents lies in the fact that they play 
an important role in the theory of weak multiplier Hopf algebras.  In a sense, any separability 
idempotent arises from a weak multiplier Hopf algebra, and having an appropriate separability 
idempotent element can guide us to construct an example of a weak multiplier Hopf algebra: 
See Proposition~3.2 of \cite{VDWangwha2}.  See \cite{VDsepid} (refer to both versions v1 and v2, 
as they are substantially different), where a systematic discussion is given on separability idempotents 
in the setting of multiplier algebras, together with some examples.  Separability idempotents also 
arise naturally from the theory of discrete quantum groups (See Section~3 of \cite{VDsepid}).

Let us now turn our attention to the case when the groupoid $G$ is equipped with a compatible 
locally compact topology.  A natural question to ask is whether we can formulate a $C^*$-algebraic 
theory similar to that of a weak multiplier Hopf algebra, together with a suitable separability idempotent 
element, keeping all the topological aspects. 

The aim of the current paper is to define, develop, and clarify the notion of a separability 
idempotent in the setting of $C^*$-algebras.  While the setting may be more general, this makes 
things a little restrictive as well, because we need the element to be compatible with the 
$C^*$-structures on $B$ and $C$, unlike in the purely algebraic case.  On the other hand, 
not every property from the purely algebraic case will carry over (for instance, the maps 
$S_B$ and $S_C$ will only have to be densely defined).  However, it is worth noting that we are 
still able to formulate below a reasonable notion.  

Such separability idempotents turn out to be quite useful in developing a $C^*$-algebraic 
counterpart to the weak multiplier Hopf algebra theory, which would provide us with a (sub)class 
of {\em locally compact quantum groupoids\/}.  This is discussed in the authors' upcoming 
work \cite{BJKVD_LSthm}, \cite{BJKVD_qgroupoid1}, \cite{BJKVD_qgroupoid2}. 

In the purely algebraic theory of separability idempotents (see \cite{VDsepid}), the defining 
conditions imply the existence of two {\em distinguished linear functionals\/}, denoted 
$\varphi_C$ and $\varphi_B$, on the algebras $C$ and $B$ respectively.  In our formulation 
below (in the $C^*$-algebra setting), we assume instead the existence of two faithful KMS weights 
$\mu$ (on $C$) and $\nu$ (on $B$), from which the other properties follow, including the densely 
defined anti-homomorphisms.  This approach is somewhat similar in philosophy to the general 
theory of locally compact quantum groups as given in \cite{KuVa}, \cite{KuVavN}, \cite{VDvN}.

We will see that the separability idempotent $E\in M(B\otimes C)$ is more or less determined 
by the $C^*$-algebra $B$ and the weight $\nu$, where $\otimes$ is now the $C^*$-algebra (spatial) 
tensor product.  Meanwhile, the existence of $E$ means that the pair $(B,\nu)$ cannot be arbitrarily 
chosen.  It is an interesting question to explore about the nature of the $C^*$-algebra $B$, 
and we will do this towards the end of the paper. 

The paper is organized as follows.  In Section~1, we gather some basic results concerning 
$C^*$-algebra weights.  The purpose here is to set the terminologies and notations to be 
used in later sections.  Our main definition and the properties are given in Sections~2 and 3. 
We introduce the notion of a separability triple and a separability idempotent in Section~2. 
As a consequence, we obtain the (densely-defined) anti-homomorphisms $\gamma_B$ and 
$\gamma_C$ (These are the maps corresponding to $S_B$ and $S_C$ in the purely algebraic 
case.).  More results and properties are collected in Section~3, including the fact that a separability 
idempotent is ``full''. 

In Section~4, some examples are considered.  In addition to the natural example coming 
from a groupoid or a weak multiplier Hopf algebra, we included an example arising from 
a certain action groupoid and another example when the base $C^*$-algebra is the algebra 
of compact operators. 

In Section~5, the separability idempotent notion is considered in the von Neumann algebra 
setting.   We show that the $C^*$-algebra approach and the von Neumann algebra approach 
are equivalent, in the sense that one can start from the $C^*$-algebra framework, then construct 
a separability idempotent in the von Neumann algebra framework, then from this, one can 
recover the idempotent and the $C^*$-algebra we began with.  

Finally, in Section~6, we explore the nature of the base $C^*$-algebra $B$, which would 
allow the existence of a separability idempotent.  We claim that $B$ has to be a postliminal 
$C^*$-algebra.

\bigskip

{\sc Acknowledgments.}

This work began during the first-named author (Byung-Jay Kahng)'s sabbatical visit to 
University of Leuven during 2012/2013.  He is very much grateful to his coauthor (Alfons 
Van Daele) and the mathematics department at University of Leuven for their warm support 
and hospitality during his stay. 

\bigskip

\section{Preliminaries on $C^*$-algebra weights}

In this section, we review some basic notations and results concerning the weights on 
$C^*$-algebras, which will be useful later.  For standard terminologies and for a more 
complete treatment on $C^*$-algebra weights, refer to \cite{Cm}, \cite{Str}, \cite{Tk2}. 
See also a nice survey given in \cite{KuVaweightC*}.

Let $A$ be a $C^*$-algebra.  A function $\psi:A^+\to[0,\infty]$ is called a weight on $A$, 
if $\psi(x+y)=\psi(x)+\psi(y)$ for all $x,y\in A^+$ and $\psi(\lambda x)=\lambda\psi(x)$ 
for all $x\in A^+$ and $\lambda\in\mathbb{R}^+$, with the convention that $0\cdot\infty=0$.

Given a weight $\psi$ on $A$, we can consider the following subspaces:
\begin{align}
{\mathfrak M}_\psi^+:&=\bigl\{a\in A^+:\psi(a)<\infty\bigr\} \notag \\
{\mathfrak N}_\psi:&=\bigl\{a\in A:\psi(a^*a)<\infty\bigr\} \notag \\
{\mathfrak M}_\psi:&={\mathfrak N}_\psi^*{\mathfrak N}_\psi
=\operatorname{span}\{y^*x:x,y\in{\mathfrak N}_\psi\}. \notag
\end{align}
It is easy to see that ${\mathfrak N}_\psi$ is a left ideal in $M(A)$, the multiplier algebra 
of $A$.  So ${\mathfrak M}_\psi\subseteq{\mathfrak N}_\psi$.  Moreover, ${\mathfrak M}_\psi$ 
is a ${}^*$-subalgebra of $A$ spanned by ${\mathfrak M}_\psi^+$, which turns out to be its 
positive part.  It is possible to naturally extend $\psi$ on ${\mathfrak M}_\psi^+$ to a map 
from ${\mathfrak M}_\psi$ into $\mathbb{C}$, which we will still denote by $\psi$.

A weight $\psi$ is said to be ``faithful'', if $\psi(a)=0$, $a\in A^+$, implies $a=0$.  We say 
a weight $\psi$ is densely-defined (or ``semi-finite''), if ${\mathfrak N}_\psi$ is dense in $A$. 
Throughout this paper, a weight (on a $C^*$-algebra $A$) will always be assumed to be 
faithful, semi-finite, and also lower semi-continuous. 

It is useful to consider the following sets, first considered by Combes \cite{Cm}:
\begin{align}
{\mathcal F}_\psi&:=\bigl\{\omega\in A^*_+:\omega(x)\le\psi(x),\forall x\in A^+\bigr\}, \notag \\
{\mathcal G}_\psi&:=\bigl\{\alpha\omega:\omega\in{\mathcal F}_\psi,{\text { for $\alpha\in(0,1)$}}\bigr\}.
\notag
\end{align}
Here $A^*$ denotes the norm dual of $A$.  Note that on ${\mathcal F}_\psi$, one can give 
a natural order inherited from $A^*_+$.  Meanwhile, ${\mathcal G}_\psi$ is a directed subset of 
${\mathcal F}_\psi$.  That is, for every $\omega_1,\omega_2\in{\mathcal G}_\psi$, there exists 
an element $\omega\in{\mathcal G}_\psi$ such that $\omega_1\le\omega$, $\omega_2\le\omega$. 
Because of this, ${\mathcal G}_\psi$ is often used as an index set (of a net).  Since $\psi$ is 
lower semi-continuous, we would have: 
$\psi(x)=\lim_{\omega\in{\mathcal G}_\psi}\bigl(\omega(x)\bigr)$, for $x\in A^+$.

One can associate to $\psi$ a GNS-construction $({\mathcal H}_\psi,\pi_\psi,\Lambda_\psi)$. Here, 
${\mathcal H}_\psi$ is a Hilbert space, $\Lambda_\psi:{\mathfrak N}_\psi\to{\mathcal H}_\psi$ is 
a linear map such that $\Lambda_\psi({\mathfrak N}_\psi)$ is dense in ${\mathcal H}_\psi$, and 
$\bigl\langle\Lambda_\psi(a),\Lambda_\psi(b)\bigr\rangle=\psi(b^*a)$ for $a,b\in{\mathfrak N}_\psi$. 
Since $\psi$ is assumed to be lower semi-continuous, we further have that $\Lambda_\psi$ is closed. 
And, $\pi_\psi$ is the GNS representation of $A$ on ${\mathcal H}_\psi$, given by $\pi_\psi(a)
\Lambda_\psi(b)=\Lambda_\psi(ab)$ for $a\in A$, $b\in{\mathfrak N}_\psi$.  The GNS representation 
is non-degenerate, and the GNS construction is unique up to a unitary transformation.

Every $\omega\in A^*$ has a unique extension to the level of the multiplier algebra $M(A)$, which 
we may still denote by $\omega$.  From this fact, it follows easily that any proper weight $\psi$ 
on $A$ has a natural extension to the weight $\bar{\psi}$ on $M(A)$.  For convenience, we will use 
the notations $\overline{\mathfrak N}_\psi={\mathfrak N}_{\bar{\psi}}$ and $\overline{\mathfrak M}_\psi
={\mathfrak M}_{\bar{\psi}}$.  Also, the GNS construction for a proper weight $\psi$ on $A$ has 
a natural extension to the GNS construction for $\bar{\psi}$ on $M(A)$, with ${\mathcal H}_{\bar{\psi}}
={\mathcal H}_\psi$.

To give somewhat of a control over the non-commutativity of $A$, one introduces the notion of 
KMS weights (see \cite{Cm} and Chapter~VIII of \cite{Tk2}).  The notion as defined below (due to 
Kustermans \cite{KuKMS}) is slightly different from the original one given by Combes, but equivalent.

\begin{defn}\label{KMSweight}
Let $\psi$ be a faithful, semi-finite, lower semi-continuous weight.  It is called a {\em KMS weight\/}, 
if there exists a norm-continuous one-parameter group of automorphisms $(\sigma_t)_{t\in\mathbb{R}}$
of $A$ such that $\psi\circ\sigma_t=\psi$ for all $t\in\mathbb{R}$, and 
$$
\psi(a^*a)=\psi\bigl(\sigma_{i/2}(a)\sigma_{i/2}(a)^*\bigr) {\text { for all $a\in{\mathcal D}
(\sigma_{i/2})$.}}
$$
Here, $\sigma_{i/2}$ is the analytic generator of the one-parameter group $(\sigma_t)$ at $z=i/2$, 
and ${\mathcal D}(\sigma_{i/2})$ is its domain.  In general, with the domain properly defined, it is 
known that $\sigma_z$, $z\in\mathbb{C}$, is a closed map.
\end{defn}

The one-parameter group $(\sigma_t)$ is called the ``modular automorphism group'' for $\psi$. 
It is uniquely determined, as $\psi$ is faithful.  In the special case when $\psi$ is a trace, 
that is, $\psi(a^*a)=\psi(aa^*)$ for $a\in{\mathfrak N}_\psi$, it is clear that $\psi$ is KMS, 
with the modular automorphism group being trivial ($\sigma\equiv\operatorname{Id}$).

Basic properties of KMS weights can be found in \cite{Tk2}, \cite{KuKMS}, \cite{KuVaweightC*}. 
In particular, there exists a (unique) anti-unitary operator $J$ (the ``modular conjugation'') 
on ${\mathcal H}_\psi$ such that $J\Lambda_\psi(x)=\Lambda_\psi\bigl(\sigma_{i/2}(x)^*\bigr)$, 
for $x\in{\mathfrak N}_\psi\cap{\mathcal D}(\sigma_{i/2})$.  There exists also a strictly positive 
operator $\nabla$ (the ``modular operator'') on ${\mathcal H}_\psi$ such that 
$\nabla^{it}\Lambda_\psi(a)=\Lambda_\psi\bigl(\sigma_t(a)\bigr)$, for $a\in{\mathfrak N}_{\psi}$ 
and $t\in\mathbb{R}$.  Following is a standard result, which can be easily extended to elements 
in the multiplier algebra:

\begin{lem}\label{lemKMS}
Let $\psi$ be a KMS weight on a $C^*$-algebra $A$, with GNS representation $({\mathcal H}_\psi,
\pi_\psi,\Lambda_\psi)$.  Then we have:
\begin{enumerate}
  \item Let $a\in{\mathcal D}(\sigma_{i/2})$ and $x\in{\mathfrak N}_\psi$. 
Then $xa\in{\mathfrak N}_\psi$ and $\Lambda_\psi(xa)=J\pi_\psi\bigl(\sigma_{i/2}(a)\bigr)^*J
\Lambda_\psi(x)$.
  \item Let $a\in{\mathcal D}(\sigma_{-i})$ and $x\in{\mathfrak M}_\psi$.  Then 
$ax, x\sigma_{-i}(a)\in{\mathfrak M}_\psi$ and $\psi(ax)=\psi\bigl(x\sigma_{-i}(a)\bigr)$.
  \item Let $x\in{\mathfrak N}_\psi\cap{\mathfrak N}_\psi^*$ and $a\in{\mathfrak N}_\psi^*\cap 
{\mathcal D}(\sigma_{-i})$ be such that $\sigma_{-i}(a)\in{\mathfrak N}_\psi$.  Then 
$\psi(ax)=\psi\bigl(x\sigma_{-i}(a)\bigr)$.
\end{enumerate}
\end{lem}

It is known (see \cite{Tk2}, \cite{KuVaweightC*}) that when we lift a KMS weight 
to the level of von Neumann algebra $\pi_\psi(A)''$ in an evident way, we obtain a normal, 
semi-finite, faithful (``n.s.f.'')~weight $\widetilde{\psi}$.  In addition, the modular 
automorphism group $(\sigma^{\widetilde{\psi}}_t)$ for the n.s.f.~weight $\widetilde{\psi}$ 
leaves the $C^*$-algebra $B$ invariant, and the restriction of $(\sigma^{\widetilde{\psi}}_t)$ 
to $B$ coincides with our $(\sigma_t)$.  The operators $J$, $\nabla$  above are none other 
than the restrictions of the corresponding operators arising in the standard Tomita--Takesaki 
theory of von Neumann algebra weights.

In Tomita--Takesaki theory, a useful role is played by the Tomita ${}^*$-algebra.  In our case, 
consider the n.s.f.~weight $\widetilde{\psi}$ on the von Neumann algebra $\pi_\psi(A)''$.  Then 
the Tomita ${}^*$-algebra for $\widetilde{\psi}$, denoted by ${\mathcal T}_{\widetilde{\psi}}$, 
is as follows:
$$
{\mathcal T}_{\widetilde{\psi}}:=\bigl\{x\in{\mathfrak N}_{\widetilde{\psi}}\cap
{\mathfrak N}_{\widetilde{\psi}}^*:{\text { $x$ is analytic w.r.t. $\sigma^{\widetilde{\psi}}$, and 
$\sigma^{\widetilde{\psi}}_z(x)\in{\mathfrak N}_{\widetilde{\psi}}\cap{\mathfrak N}_{\widetilde{\psi}}^*$, 
$\forall z\in\mathbb{C}$}}\bigr\}.
$$
It is a strongly ${}^*$-dense subalgebra in $\pi_\psi(A)''$, and it satisfies many useful properties. 
Refer to the standard textbooks on the modular theory, for instance \cite{Tk2}, \cite{Str}. 
Furthermore, it is known that ${\mathcal T}_{\psi}:={\mathcal T}_{\widetilde{\psi}}\cap A$ is 
(norm)-dense in the $C^*$-algebra $A$, so turns out to be useful as we work with the weight 
$\psi$ in the $C^*$-algebra setting.  In particular, it is easy to see that for every $z\in\mathbb{C}$, 
the domains ${\mathcal D}(\sigma_z)$, which contain ${\mathcal T}_{\psi}$, are dense in $A$.

\section{Definition of a separability idempotent}

In this section, we wish to give the proper definition of a separability idempotent element in the 
$C^*$-algebra framework.  Let us begin by considering a triple $(E,B,\nu)$, where 
\begin{itemize}
  \item $B$ is a $C^*$-algebra,
  \item $\nu$ is a (faithful) KMS weight on $B$, together with its associated norm-continuous automorphism 
group $(\sigma^{\nu}_t)$,
  \item $E\in M(B\otimes C)$ is a self-adjoint idempotent element (so a projection, satisfying $E^2=E=E^*$), 
for some $C^*$-algebra $C$ such that there exists a ${}^*$-anti-isomorphism $R:B\to C$.
\end{itemize}

\begin{rem}
In the above, we can see immediately that $C\cong B^{\operatorname{op}}$ as $C^*$-algebras.  Even so, 
there is no reason to assume that $C$ and $B^{\operatorname{op}}$ to be exactly same.  Going forward, 
our triple $(E,B,\nu)$ will be understood in such a way that the $C^*$-algebra $C$ and the ${}^*$-anti-isomorphism 
$R:B\to C$ are implicitly fixed, with $C$ possibly different from $B^{\operatorname{op}}$.
\end{rem}

The weight theory says ${\mathfrak N}_\nu$ is only a left ideal in $B$, but we have below a useful 
result involving $B_0={\mathcal D}(\sigma^{\nu}_{i/2})$, which is dense in $B$.

\begin{prop}\label{B0}
For any $x\in{\mathfrak N}_\nu$ and $b\in B_0$, we have: $xb
\in{\mathfrak N}_\nu$.  As a consequence, we have:
$$
{\mathfrak N}_\nu B_0\subseteq{\mathfrak N}_\nu,\quad {\mathfrak M}_\nu B_0\subseteq{\mathfrak M}_\nu,
\quad B_0^*{\mathfrak M}_\nu\subseteq{\mathfrak M}_\nu.
$$
\end{prop}

\begin{proof}
See (1) of Lemma~\ref{lemKMS}, from which it follows that $\nu\bigl((xb)^*(xb)\bigr)=
\bigl\langle\Lambda_{\nu}(xb),\Lambda_{\nu}(xb)\bigr\rangle\le\bigl\|\sigma^{\nu}_{i/2}(b)\bigr\|^2
\nu(x^*x)$.  The result ${\mathfrak N}_\nu B_0\subseteq{\mathfrak N}_\nu$ is an immediate 
consequence.  As for the next two results, use the fact that ${\mathfrak M}_\nu
={\mathfrak N}_\nu^*{\mathfrak N}_\nu$.
\end{proof}

Consider the triple $(E,B,\nu)$ as above, with $C$ and $R$ understood.  We now give the definition of the 
separability triple:

\begin{defn}\label{separabilitytriple}
We will say that $(E,B,\nu)$ is a {\em separability triple\/}, if the following conditions hold:
\begin{enumerate}
  \item $(\nu\otimes\operatorname{id})(E)=1$
  \item For $b\in B_0$, we have: $(\nu\otimes\operatorname{id})\bigl(E(b\otimes1)\bigr)
=R\bigl(\sigma^{\nu}_{i/2}(b)\bigr)$.
\end{enumerate}
If $(E,B,\nu)$ forms a separability triple, then we say $E$ is a {\em separability idempotent\/}.
\end{defn}

\begin{rem}
In Definition~\ref{separabilitytriple} above, (1) means that for any $\omega\in C^*_+$, 
we require $(\operatorname{id}\otimes\omega)(E)\in\overline{\mathfrak M}_\nu\,\bigl(\subseteq 
M(B)\bigr)$ and that $\nu\bigl((\operatorname{id}\otimes\omega)(E)\bigr)=\omega(1)$.  From this, 
it will follow that $(\operatorname{id}\otimes\omega)(E)\in\overline{\mathfrak M}_\nu$ for any 
$\omega\in C^*$, and that $\nu\bigl((\operatorname{id}\otimes\omega)(E)\bigr)=\omega(1),
\forall\omega\in C^*$.  We are using here a natural extension of $\nu$ to the multiplier 
algebra $M(B)$, still denoted by $\nu$.  However, we will soon see that $(\operatorname{id}
\otimes\omega)(E)\in B$, actually.  See Corollary following Proposition~\ref{BtensorC}.

By Proposition~\ref{B0}, we know $(\operatorname{id}\otimes\omega)(E)\,b\in{\mathfrak M}_\nu,
\forall\omega\in C^*$ and $\forall b\in B_0$.  What condition~(2) of the definition is saying is that we 
further require 
$\nu\bigl((\operatorname{id}\otimes\omega)(E)\,b\bigr)=\omega\bigl((R\circ\sigma^{\nu}_{i/2})(b)\bigr)$.
\end{rem}

The following result provides sort of a uniqueness result for the separability idempotent $E$.

\begin{prop}\label{uniqueE}
Let $E$ be a separability idempotent.  Then $E$ is uniquely determined by the data $(B,\nu,R)$.
\end{prop}

\begin{proof}
From Definition~\ref{separabilitytriple} and Remark following it, we know that 
\begin{equation}\label{(Edefeq)}
\bigl(\,\nu(\cdot\,b)\otimes\omega\bigr)(E)=\omega\bigl((R\circ\sigma^{\nu}_{i/2})(b)\bigr),
\quad\forall b\in B_0,\ \forall\omega\in C^*.
\end{equation}
Since $B_0$ is dense in $B$, we know $\bigl\{\nu(\,\cdot\,b):b\in B_0\bigr\}$ is dense in $B^*$. 
This shows that $E$ is completely determined by the data $(B,\nu,R)$.
\end{proof}

As we regard a separability triple $(E,B,\nu)$ to implicitly fix the $C^*$-algebra $C$ and the 
anti-isomorphism $R$ (see an earlier remark), we can see from the above result that the $E$ 
is more or less characterized by the pair $(B,\nu)$.  However, the element $E\in M(B\otimes C)$ 
that is determined by $(B,\nu)$ via equation~\eqref{(Edefeq)} does not necessarily have to be 
a projection.  In other words, the existence of the separability idempotent $E$ is really 
a condition on the pair $(B,\nu)$.  In Section 6 below, we will discuss a little about the nature 
of the $C^*$-algebra $B$ that could allow this to happen.

Let us now gather some basic properties of the separability triple $(E,B,\nu)$ and the 
separability idempotent $E\in M(B\otimes C)$.  For convenience, let us write $\gamma_B(b)
:=(R\circ\sigma^{\nu}_{i/2})(b)$, for $b\in B_0$, which determines a map $\gamma_B:B_0\to C$. 
Since $\sigma^{\nu}_{i/2}$ is a closed map (see comment given in Definition~\ref{KMSweight}), 
we see that $\gamma_B$ is also a closed map.

\begin{prop}
Considered as a map from $B$ to $C$, the aforementioned map $\gamma_B$ is 
a closed, densely-defined, injective map, whose range is dense in $C$.  Moreover, 
for $b\in B_0$, we have:
$(\nu\otimes\operatorname{id})\bigl(E(b\otimes1)\bigr)=\gamma_B(b)$.
\end{prop}

\begin{proof}
Since ${\mathcal D}(\gamma_B)=B_0$ is dense in $B$, it is clear that $\gamma_B$ is 
densely-defined.  We have already seen that it is closed.  The injectivity of $\gamma_B$ is 
easy to see, because $R$ is an anti-isomorphism and $\sigma^{\nu}_t$ is an automorphism 
for all $t$. 

By the property of the Tomita algebra, we know that ${\mathcal T}_{\nu}\subseteq B_0$ 
and $\sigma^{\tilde{\nu}}_{i/2}({\mathcal T}_{\tilde{\nu}})={\mathcal T}_{\tilde{\nu}}$.  So we have: 
$\sigma^{\nu}_{i/2}({\mathcal T}_{\nu})=\sigma^{\nu}_{i/2}({\mathcal T}_{\tilde{\nu}}\cap B)
={\mathcal T}_{\tilde{\nu}}\cap B={\mathcal T}_{\nu}$, which is dense in $B$.  In addition, 
since $R$ is an anti-isomorphism between $B$ and $C$, we see that $\operatorname{Ran}
(\gamma_B)=R\bigl(\sigma^{\nu}_{i/2}(B_0)\bigr)\supseteq R({\mathcal T}_{\nu})$, which is 
dense in $C$.

The last result is just re-writing (2) of Definition~\ref{separabilitytriple}.
\end{proof}

For convenience, write $C_0=\operatorname{Ran}(\gamma_B)$, which is dense in $C$. 
By the above proposition, we can consider $\gamma_B^{-1}:C_0\to B$.  Viewed as a map from $C$ 
to $B$, it is clear that the map $\gamma_B^{-1}$ is also a closed, densely-defined injective map, 
having a dense range.  In fact, we will have: $\gamma_B^{-1}=\sigma^{\nu}_{-i/2}\circ R^{-1}$, and 
${\mathcal D}(\gamma_B^{-1})=\operatorname{Ran}(\gamma_B)=C_0$, while $\operatorname{Ran}
(\gamma_B^{-1})={\mathcal D}(\gamma_B)=B_0$.

\begin{prop}\label{gammagamma-1}
\begin{enumerate}
  \item $B_0={\mathcal D}(\gamma_B)$ is closed under multiplication, and $\gamma_B$ is an 
anti-homomorphism.
  \item $C_0={\mathcal D}(\gamma_B^{-1})$ is closed under multiplication, and $\gamma_B^{-1}$ 
is an anti-homomorphism.
\end{enumerate}
\end{prop}

\begin{proof}
Suppose $b,b'\in B_0={\mathcal D}(\sigma^{\nu}_{i/2})$.  Being analytic elements, we have 
$bb'\in B_0$.  Meanwhile, since $\gamma_B=R\circ\sigma^{\nu}_{i/2}$, with $\sigma^{\nu}_{i/2}$ 
being an automorphism and $R$ an anti-isomorphism, we have $\gamma_B(bb')=\gamma_B(b')
\gamma_B(b)$. Similarly, we have: $\gamma_B^{-1}(cc')=\gamma_B^{-1}(c')\gamma_B^{-1}(c)$ 
for $c,c'\in C_0$.
\end{proof}

Here is a property that relates the $C^*$-algebras $B$ and $C$, at the level of their dense 
subspaces:
\begin{prop}\label{sepidgamma}
\begin{enumerate}
  \item For $b\in B_0$, we have: $E(b\otimes1)=E\bigl(1\otimes\gamma_B(b)\bigr)$.
  \item For $c\in C_0$, we have: $E(1\otimes c)=E\bigl(\gamma_B^{-1}(c)\otimes1\bigr)$.
\end{enumerate}
\end{prop}

\begin{proof}
(1). For $b,b'\in B_0$, 
\begin{align}
\bigl(\nu(\,\cdot\,b')\otimes\operatorname{id}\bigr)\bigl(E(b\otimes1)\bigr)&=
(\nu\otimes\operatorname{id})\bigl(E(bb'\otimes1)\bigr)
=\gamma_B(bb')=\gamma_B(b')\gamma_B(b) \notag \\
&=(\nu\otimes\operatorname{id})\bigl(E(b'\otimes1)\bigr)\gamma_B(b)
=(\nu\otimes\operatorname{id})\bigl(E(b'\otimes\gamma_B(b))\bigr)  \notag \\
&=(\nu\otimes\operatorname{id})\bigl(E(1\otimes\gamma_B(b))(b'\otimes1)\bigr) \notag \\
&=\bigl(\nu(\,\cdot\,b')\otimes\operatorname{id}\bigr)\bigl(E(1\otimes\gamma_B(b))\bigr).  \notag 
\end{align}
In the second equation, we are using (2) of Definition~\ref{separabilitytriple}.  The third 
equation holds because $\gamma_B$ is an anti-homomorphism.  Since $\nu$ is faithful, 
and since the result is true for all $b'\in B_0$, which is dense in $B$, we conclude that 
$E(b\otimes1)=E\bigl(1\otimes\gamma_B(b)\bigr)$.

(2). This is an immediate consequence of (1), because $C_0=\operatorname{Ran}(\gamma_B)$ 
and $B_0=\operatorname{Ran}(\gamma_B^{-1})$.
\end{proof}

From $(\nu\otimes\operatorname{id})\bigl(E(b\otimes1)\bigr)=R\bigl(\sigma^{\nu}_{i/2}(b)\bigr)$, 
$b\in B_0$, take the adjoint.  Knowing that $R$ is a ${}^*$-anti-isomorphism, we have:
$$
(\nu\otimes\operatorname{id})\bigl((b^*\otimes1)E\bigr)=\bigl[R(\sigma^{\nu}_{i/2}(b))\bigr]^*
=R\bigl([\sigma^{\nu}_{i/2}(b)]^*\bigr)=R\bigl(\sigma^{\nu}_{-i/2}(b^*)\bigr).
$$
This observation means that for $b\in B_0^*={\mathcal D}(\sigma^{\nu}_{-i/2})$, which is also 
dense in $B$, the expression $(\nu\otimes\operatorname{id})\bigl((b\otimes1)E\bigr)$ is valid, 
and $(\nu\otimes\operatorname{id})\bigl((b\otimes1)E\bigr)=(R\circ\sigma^{\nu}_{-i/2})(b)$. 
Or, put another way, we have: 
$$
\nu\bigl(b\,(\operatorname{id}\otimes\omega)(E)\bigr)
=\omega\bigr((R\circ\sigma^{\nu}_{-i/2})(b)\bigr),\quad{\text { for $\omega\in C^*$, 
$b\in{\mathcal D}(\sigma^{\nu}_{-i/2})$.}}
$$

So, by the same argument as in the case of the map $\gamma_B$, it is clear that the map 
$b\mapsto(R\circ\sigma^{\nu}_{-i/2})(b)$ is closed and densely-defined on $B$, injective, 
and has a dense range in $C$.  Also, its inverse map $c\to(\sigma^{\nu}_{i/2}\circ R)(c)$ is 
closed and densely-defined in $C$, injective, and has a dense range in $B$.  So let us define 
the maps $\gamma_C:C\to B$ and $\gamma_C^{-1}:B\to C$ as follows:

\begin{prop}\label{sepidgamma'}
Write $\gamma_C:=\sigma^{\nu}_{i/2}\circ R^{-1}$ and $\gamma_C^{-1}:=R\circ\sigma^{\nu}_{-i/2}$. 
Then: 
\begin{enumerate}
  \item The map $\gamma_C$ is closed and densely-defined on $C$, injective, and has a dense 
range in $B$.
  \item The map $\gamma_C^{-1}$ is closed and densely-defined on $B$, injective, and has a dense 
range in $C$.
  \item Both maps $\gamma_C$ and $\gamma_C^{-1}$ are anti-homomorphisms:  That is, 
$\gamma_C(cc')=\gamma_C(c')\gamma_C(c)$, $c,c'\in{\mathcal D}(\gamma_C)$, and 
$\gamma_C^{-1}(bb')=\gamma_C^{-1}(b')\gamma_C^{-1}(b)$, $b,b'\in{\mathcal D}(\gamma_C^{-1})$.
  \item For $c\in{\mathcal D}(\gamma_C)$, we have: 
$(1\otimes c)E=\bigl(\gamma_C(c)\otimes1\bigr)E$.
  \item For $b\in{\mathcal D}(\gamma_C^{-1})$, we have: 
$(b\otimes1)E=\bigl(1\otimes\gamma_C^{-1}(b)\bigr)E$.
\end{enumerate}
\end{prop}

\begin{proof}
(1) and (2) were already observed in the previous paragraph.  Proof of (3) is done exactly as in 
Proposition~\ref{gammagamma-1}.  The anti-homomorphism property is needed in the proof of the results 
(4) and (5), which is done using a similar argument as in Proposition~\ref{sepidgamma}.
\end{proof}

\begin{rem}
The results of Proposition~\ref{sepidgamma} and Proposition~\ref{sepidgamma'} suggest us that $E$ 
behaves very much like a separability idempotent in the purely algebraic setting \cite{VDsepid}, 
justifying our definition given in Definition~\ref{separabilitytriple}. Our maps $\gamma_B$ and 
$\gamma_C$ correspond to the maps $S_B$ and $S_C$ in \cite{VDsepid}.  Indeed, our 
Proposition~\ref{sepidgamma}\,(1) and Proposition~\ref{sepidgamma'}\,(4) are exactly the 
defining axioms in the algebraic case (see Definition~1.4 of \cite{VDsepid}), for $S_B$ and $S_C$. 
The only difference is that in the algebraic case, the maps $S_B$ and $S_C$ are bijections, while 
our $\gamma_B$, $\gamma_C$ maps are densely-defined with dense ranges.
\end{rem}

Compared to the purely algebraic setting, we have here the ${}^*$-structure.  The next result 
is about the relationship between the ${}^*$-structure and the maps $\gamma_B$, $\gamma_C$. 
Observe that in general, the maps $\gamma_B$, $\gamma_C$ are not necessarily ${}^*$-maps.

\begin{prop}\label{gamma*}
\begin{enumerate}
  \item For $b\in{\mathcal D}(\gamma_B)$, we have: $\gamma_B(b)^*\in{\mathcal D}(\gamma_C)$, 
and $\gamma_C\bigl(\gamma_B(b)^*\bigr)^*=b$.
  \item Similarly for $c\in{\mathcal D}(\gamma_C)$, we have $\gamma_B\bigl(\gamma_C(c)^*\bigr)^*
  =c$.
\end{enumerate}
\end{prop}

\begin{proof}
(1). If $b\in{\mathcal D}(\gamma)={\mathcal D}(\sigma^{\nu}_{i/2})$, then we know from 
Proposition~\ref{sepidgamma} that $E(b\otimes1)=E\bigl(1\otimes\gamma_B(b)\bigr)$.  Taking 
adjoints, we have: $(b^*\otimes1)E=\bigl(1\otimes\gamma_B(b)^*\bigr)E$. 

Note that $b^*\in{\mathcal D}(\sigma^{\nu}_{-i/2})=D(\gamma_C^{-1})$. 
Comparing with Proposition~\ref{sepidgamma'}, it follows that $\gamma_B(b)^*
\in{\mathcal D}(\gamma_C)$ and $\gamma_C\bigl(\gamma_B(b)^*\bigr)=b^*$.  This is equivalent 
to saying that $\gamma_C\bigl(\gamma_B(b)^*\bigr)^*=b$.

(2). Similar argument will show that: $\gamma_B\bigl(\gamma_C(c)^*\bigr)^*=c$, 
$c\in{\mathcal D}(\gamma_C)$.
\end{proof}

Notice the properties of the maps $\gamma_B$, $\gamma_C$ obtained so far, namely, being 
closed and densely-defined anti-homomorphisms satisfying the result like Proposition~\ref{gamma*} 
for the ${}^*$-structure.  One may observe that the behavior of the maps $\gamma_B$, $\gamma_C$ 
resemble that of the antipode map for a locally compact quantum group.  Indeed, in a certain context, 
it turns out that the maps $\gamma_B$, $\gamma_C$ are indeed the antipode map of a locally 
compact quantum groupoid $(A,\Delta)$, restricted to the subalgebras $B$ and $C$, which are 
essentially the ``source algebra'' and the ``target algebra''.  And, our anti-isomorphism $R$ would 
be the restriction to $B$ of the ``unitary antipode'' $R_A:A\to A$.  More systematic treatment on 
this direction will be given in our upcoming works \cite{BJKVD_LSthm}, \cite{BJKVD_qgroupoid1}, 
\cite{BJKVD_qgroupoid2}.

\section{Properties of the separability idempotent}

Suppose we have a separability triple $(E,B,\nu)$ in the sense of the previous section, 
with $C$ and $R$ understood.  We have already observed that $E$ is uniquely determined by 
the pair $(B,\nu)$.  Now, using the ${}^*$-anti-isomorphism $R:B\to C$, we can also define a faithful 
weight $\mu$ on $C$, by
$$
\mu:=\nu\circ R^{-1}=\nu\circ\gamma_C.
$$
The second characterization is true because $\gamma_C=\sigma^{\nu}_{i/2}\circ R^{-1}$, 
while $\nu$ is $\sigma^{\nu}$-invariant.  It is not difficult to show that $\mu$ is also a KMS 
weight on the $C^*$-algebra $C$, together with its modular automorphism group 
$(\sigma^{\mu}_t)_{t\in\mathbb{R}}$, given by $\sigma^{\mu}_t:=R\circ\sigma^{\nu}_{-t}\circ R^{-1}$. 
It turns out that the pair $(C,\mu)$ behaves a lot like $(B,\nu)$.

\begin{prop}\label{mudefnprop}
Let $\mu$ be as above.  Then we have:
$$
(\operatorname{id}\otimes\mu)(E)=1.
$$
\end{prop}

\begin{proof}
As in Remark following Definition~\ref{separabilitytriple}, the above equation means that 
$(\theta\otimes\operatorname{id})(E)\in\overline{\mathfrak M}_{\mu}$ for all $\theta\in B^*$ 
(naturally extended to the multiplier algebra level), and that $\mu\bigl((\theta\otimes\operatorname{id})(E)\bigr)
=\theta(1)$.

We can verify this for $\theta=\nu(\,\cdot\,b)\in B^*$, where $b\in{\mathcal D}(\gamma_B)$.  Such 
functionals are dense in $B^*$, so that will prove the proposition.  To see this, consider an arbitrary  
$b\in{\mathcal D}(\gamma_B)$.  Then by Definition~\ref{separabilitytriple}\,(2), we have 
$(\nu\otimes\operatorname{id})\bigl(E(b\otimes1)\bigr)=(R\circ\sigma^{\nu}_{i/2})(b)=\gamma_B(b)$. 
Apply here $\mu$.  Then we have: 
$$
\mu\bigl((\nu\otimes\operatorname{id})(E(b\otimes1))\bigr)=\mu\bigl((R\circ\sigma^{\nu}_{i/2})(b)\bigr)
=\nu(b),
$$
where, we used the fact that $\mu=\nu\circ R^{-1}$ and that $\nu$ is $\sigma^{\nu}$-invariant.  Observe that 
this equation can be re-written as 
$$
\mu\bigl((\theta\otimes\operatorname{id})(E)\bigr)=\theta(1),
$$
proving the claim.
\end{proof}

\begin{cor}
As a consequence of the previous proposition, we have the following:
$$
(\operatorname{id}\otimes\mu)\bigl((1\otimes c)E\bigr)=\gamma_C(c),
$$
for $c\in{\mathcal D}(\gamma_C)$.
\end{cor}

\begin{proof}
Use the fact that $(1\otimes c)E=\bigl(\gamma_C(c)\otimes1\bigr)E$, for $c\in D(\gamma_C)$ and that 
$(\operatorname{id}\otimes\mu)(E)=1$.  Note also that ${\mathcal D}(\gamma_C)
={\mathcal D}\bigl((\sigma^{\nu}_{i/2}\circ R^{-1})\bigr)={\mathcal D}(\sigma^{\mu}_{-i/2})$.
\end{proof}

We also have the following alternative descriptions for the maps $\gamma_B$, $\gamma_B^{-1}$, 
$\gamma_C$, $\gamma_C^{-1}$, which will be useful later:

\begin{prop}\label{gammagamma'eqns}
We have:
\begin{align}
&\gamma_B=R\circ\sigma^{\nu}_{i/2}=\sigma^{\mu}_{-i/2}\circ R \notag \\
&\gamma_B^{-1}=\sigma^{\nu}_{-i/2}\circ R^{-1}=R^{-1}\circ\sigma^{\mu}_{i/2} \notag \\
&\gamma_C=\sigma^{\nu}_{i/2}\circ R^{-1}=R^{-1}\circ\sigma^{\mu}_{-i/2} \notag \\
&\gamma_C^{-1}=R\circ\sigma^{\nu}_{-i/2}=\sigma^{\mu}_{i/2}\circ R \notag
\end{align}
\end{prop}

\begin{proof}
The definitions involving the weight $\nu$ have been already observed in the previous section. 
As for the descriptions involving the weight $\mu$, we just use the fact that 
$\sigma^{\mu}_{i/2}=R\circ\sigma^{\nu}_{-i/2}\circ R^{-1}$ and $\sigma^{\mu}_{-i/2}=
R\circ\sigma^{\nu}_{i/2}\circ R^{-1}$.
\end{proof}

In the below is the result that is usually assumed as part of the definition of a separability 
idempotent in the purely algebraic setting:

\begin{prop}\label{BtensorC}
For all $b\in B$ and all $c\in C$, we have:
$$
E(1\otimes c)\in B\otimes C,\quad (1\otimes c)E\in B\otimes C,\quad
(b\otimes 1)E\in B\otimes C,\quad E(b\otimes 1)\in B\otimes C.
$$
Here, $\otimes$ is the (spatial) $C^*$-tensor product.
\end{prop}

\begin{proof}
Since $E\in M(B\otimes C)$, we know that $E(b\otimes c)\in B\otimes C$, for any $b\in B$ and 
any $c\in C$.  If, in particular, $b_0\in{\mathcal D}(\gamma_B)$, we have:
$$
E\bigl(1\otimes\gamma_B(b_0)c\bigr)=E\bigl(1\otimes\gamma_B(b_0)\bigr)(1\otimes c)
=E(b_0\otimes1)(1\otimes c)=E(b_0\otimes c)\in B\otimes C.
$$
But the elements $\gamma_B(b_0)c$, for $b_0\in{\mathcal D}(\gamma_B)$, $c\in C$, 
are dense in $C$, because $\operatorname{Ran}(\gamma_B)$ is dense in $C$ and 
$C^2=C$.  It follows that $E(1\otimes c)\in B\otimes C$ for all $c\in C$.  Similar argument 
holds for each of the other three cases.
\end{proof}

\begin{cor}
As a consequence of the proposition, we have: 
$$
(\operatorname{id}\otimes\omega)(E)\in B,\quad\forall\omega\in C^*,\quad
{\text { and }}\quad
(\theta\otimes\operatorname{id})(E)\in C,\quad\forall\theta\in B^*.
$$
\end{cor}

\begin{proof}
Let $b\in{\mathcal D}(\gamma_B)\subseteq B$.  Then, we know from the above proposition that 
$E(b\otimes1)\in B\otimes C$. So
$$
(\nu(\,\cdot\,b)\otimes\operatorname{id})(E)=(\nu\otimes\operatorname{id})\bigl(E(b\otimes1))\in C.
$$
Since the $\nu(\,\cdot\,b)$ form a dense subspace of $B^*$, this means that $(\theta\otimes
\operatorname{id})(E)\in C$, for all $\theta\in B^*$.  Similarly, we have $(\operatorname{id}
\otimes\omega)(E)\in B$, for all $\omega\in C^*$.
\end{proof}

We already knew that $(\operatorname{id}\otimes\omega)(E)\in M(B)$, but we now see that 
it is actually contained in $B$.  Moreover, we see below that such elements form a dense 
subspace in $B$.  Similar for $C$.

\begin{prop}\label{Efull}
The separability idempotent $E$ is ``full'', in the sense that

$\bigl\{(\theta\otimes\operatorname{id})(E(b\otimes1)):b\in B,\theta\in B^*\bigr\}$ is dense in $C$,

$\bigl\{(\theta\otimes\operatorname{id})((b\otimes1)E):b\in B,\theta\in B^*\bigr\}$ is dense in $C$,

$\bigl\{(\operatorname{id}\otimes\omega)((1\otimes c)E):c\in C,\omega\in C^*\bigr\}$ is dense in $B$,

$\bigl\{(\operatorname{id}\otimes\omega)(E(1\otimes c)):c\in C,\omega\in C^*\bigr\}$ is dense in $B$.
\end{prop}

\begin{proof}
By Corollary following Proposition~\ref{BtensorC}, we are assured that containment statement for 
each set is true.  So we just need to show that they are dense subsets.

Consider arbitrary $b_1,b_2\in{\mathcal D}(\gamma_B)$, which is dense in $B$.  We know from 
Proposition~\ref{gammagamma-1} that ${\mathcal D}(\gamma_B)$ is closed under multiplication, 
so $b_1b_2\in{\mathcal D}(\gamma_B)$.  For $\theta=\nu(\,\cdot\,b_2)\in B^*$, we have:
$$
(\theta\otimes\operatorname{id})\bigl(E(b_1\otimes1)\bigr)=(\nu\otimes\operatorname{id})
\bigl(E(b_1b_2\otimes1)\bigr)=\gamma_B(b_1b_2).
$$
This shows that $\bigl\{(\theta\otimes\operatorname{id})(E(b\otimes1)):b\in B,\theta\in B^*\bigr\}$ 
contains $\bigl\{\gamma_B(b_1b_2):b_1,b_2\in{\mathcal D}(\gamma_B)\bigr\}$, which is dense 
in $C$ because $\gamma_B$ was shown to have a dense range in $C$.  This proves the first 
statement.  Proofs for the other statements are very much similar, knowing that the maps 
$\gamma_B$, $\gamma_B^{-1}$, $\gamma_C$, $\gamma_C^{-1}$ all have dense ranges.
\end{proof}

Here is a result that is related to the fact that $E$ is full.  While it is possible to give the proof 
using the above proposition, we instead chose to give a direct proof, which seems simpler.

\begin{prop}
\begin{enumerate}
  \item If $(1\otimes c)E=0$, $c\in C$, then necessarily $c=0$.
  \item If $E(1\otimes c)=0$, $c\in C$, then necessarily $c=0$.
  \item If $E(b\otimes 1)=0$, $b\in B$, then necessarily $b=0$.
  \item If $(b\otimes 1)E=0$, $b\in B$, then necessarily $b=0$.
\end{enumerate}
\end{prop}

\begin{proof}
(1). Let $c\in C$ be such that $(1\otimes c)E=0$.  While, for any $b\in{\mathcal D}(\gamma_B)$, 
we know $\gamma_B(b)=(\nu\otimes\operatorname{id})\bigl(E(b\otimes1)\bigr)$.  So we have: 
$$
c\gamma_B(b)=c(\nu\otimes\operatorname{id})\bigl(E(b\otimes1)\bigr)
=(\nu\otimes\operatorname{id})\bigl((1\otimes c)E(b\otimes1)\bigr)=0.
$$
Since $\operatorname{Ran}(\gamma_B)$ is dense in $C$, this means that $c=0$.

The results (2), (3), (4) can be proved similarly.
\end{proof}

\begin{rem}
The ``fullness'' of $E\in M(B\otimes C)$, as given in Proposition~\ref{Efull}, means that the left 
leg of $E$ is $B$ and the right leg of $E$ is $C$.  In the purely algebraic setting, the fullness 
of $E$ was part of the definition of $E$ being a separability idempotent \cite{VDsepid}.  There, 
from the defining axioms, one obtains the existence of certain ``distinguished linear functionals'', 
$\varphi_B$ and $\varphi_C$.  In our setting, however, we begin first with the weights $\mu$ and $\nu$, 
then obtain the results on $E$, including its fullness, the maps $\gamma_B$, $\gamma_C$, and 
the like.
\end{rem}

In the next proposition, we see that the analytic generators $\sigma^{\mu}_{-i}$ and $\sigma^{\nu}_{-i}$ 
can be characterized in terms of the maps $\gamma_B$ and $\gamma_C$.

\begin{prop}\label{modularauto}
We have:
\begin{enumerate}
  \item $\sigma^{\mu}_{-i}(c)=(\gamma_B\circ\gamma_C)(c)$, for $c\in{\mathcal D}
(\gamma_B\circ\gamma_C)$.
  \item $\sigma^{\nu}_{-i}(b)=(\gamma_B^{-1}\circ\gamma_C^{-1})(b)$, for $b\in{\mathcal D}
(\gamma_B^{-1}\circ\gamma_C^{-1})$.  
\end{enumerate}
\end{prop}

\begin{proof}
For (1), use $\gamma_B=\sigma^{\mu}_{-i/2}\circ R$ and $\gamma_C=R^{-1}\circ\sigma^{\mu}_{-i/2}$, 
observed in Proposition~\ref{gammagamma'eqns}.  For (2), use $\gamma_B^{-1}
=\sigma^{\nu}_{-i/2}\circ R^{-1}$ and $\gamma_C^{-1}=R\circ\sigma^{\nu}_{-i/2}$, again 
from Proposition~\ref{gammagamma'eqns}. 
\end{proof}

\begin{rem}
The significance of Proposition~\ref{modularauto} is that the maps $\gamma_B\circ\gamma_C$ 
and $\gamma_B^{-1}\circ\gamma_C^{-1}$ provide ``modular automorphisms'' for the weights 
$\mu$ and $\nu$, respectively.  To be more precise, recall Lemma~\ref{lemKMS}. 
Our result says that for $c\in{\mathcal D}(\gamma_B\circ\gamma_C)
={\mathcal D}(\sigma^{\mu}_{-i})$ and $x\in{\mathfrak M}_{\mu}$, we have: 
$cx,x\sigma^{\mu}_{-i}(c)\in{\mathfrak M}_{\mu}$, and $\mu(cx)=\mu\bigl(x\sigma^{\mu}_{-i}(c)\bigr)
=\mu\bigl(x(\gamma_B\circ\gamma_C)(c)\bigr)$.  Similar also for the weight $\nu$.
\end{rem}

We will conclude this section by showing a few different characterizations of the idempotent $E$.

\begin{prop}\label{sigmasigmaE}
For any $t\in\mathbb{R}$, we have: $(\sigma^{\nu}_t\otimes\sigma^{\mu}_{-t})(E)=E$.
\end{prop}

\begin{proof}
Suppose $b\in{\mathcal D}(\sigma^{\nu}_{i/2})={\mathcal D}(\gamma_B)$ be arbitrary, and compute: 
\begin{align}
&(\nu\otimes\operatorname{id})\bigl((\sigma^{\nu}_t\otimes\sigma^{\mu}_{-t})(E)(b\otimes1)\bigr)
=(\nu\otimes\operatorname{id})\bigl((\sigma^{\nu}_t\otimes\sigma^{\mu}_{-t})
[E(\sigma^{\nu}_{-t}(b)\otimes1)]\bigr)   \notag \\
&=(\nu\otimes\operatorname{id})\bigl((\operatorname{id}\otimes\sigma^{\mu}_{-t})
[E(\sigma^{\nu}_{-t}(b)\otimes1)]\bigr)   
=\sigma^{\mu}_{-t}\bigl((\nu\otimes\operatorname{id})[E(\sigma^{\nu}_{-t}(b)\otimes1)]\bigr)
\notag \\
&=\sigma^{\mu}_{-t}\bigl(\gamma_B(\sigma^{\nu}_{-t}(b))\bigr)
=\sigma^{\mu}_{-t}\bigl((R\circ\sigma^{\nu}_{i/2}\circ\sigma^{\nu}_{-t})(b)\bigr)
=\sigma^{\mu}_{-t}\bigl((R\circ\sigma^{\nu}_{-t}\circ\sigma^{\nu}_{i/2})(b)\bigr)  \notag \\
&=(R\circ\sigma^{\nu}_t\circ R^{-1})\bigl((R\circ\sigma^{\nu}_{-t}\circ\sigma^{\nu}_{i/2})(b)\bigr)
=(R\circ\sigma^{\nu}_{i/2})(b)=\gamma_B(b)  \notag \\
&=(\nu\otimes\operatorname{id})\bigl(E(b\otimes1)\bigr).
\notag
\end{align}
In the first equality, we are using the fact that $\sigma^{\nu}_t$ is an automorphism.
In the second equality, we used $\nu\circ\sigma^{\nu}_t=\nu$. 

This is true for any $b\in{\mathcal D}(\sigma^{\nu}_{i/2})$.  By the uniqueness property observed in 
Proposition~\ref{uniqueE}, we conclude that $(\sigma^{\nu}_t\otimes\sigma^{\mu}_{-t})(E)=E$, 
for any $t\in\mathbb{R}$.
\end{proof}

Write $\sigma$ to denote the flip map on $M(B\otimes C)$.  So we will have $\sigma E
\in M(C\otimes B)$.  In the below, we wish to show that $(\gamma_C\otimes\gamma_B)(\sigma E)
=E$.  However, as of now, we do not know if the expression $(\gamma_C\otimes\gamma_B)(\sigma E)$ 
even makes sense as a bounded element.  To make sense of all this, and anticipating other future 
applications, we prove first the following lemma:

\begin{lem}\label{thelemma}
Suppose $b\in{\mathcal T}_{\nu}$ and $c\in{\mathcal T}_{\mu}$.  We have:
\begin{align}
&(\gamma_C\otimes\gamma_B)\bigl((\gamma_C^{-1}(b)\otimes\gamma_B^{-1}(c))(\sigma E)\bigr)
=E(b\otimes c),  \notag \\
&(\gamma_C\otimes\gamma_B)\bigl((\sigma E)(\gamma_C^{-1}(b)\otimes\gamma_B^{-1}(c))\bigr)
=(b\otimes c)E,  \notag \\
&(\gamma_C\otimes\gamma_B)\bigl((1\otimes\gamma_B^{-1}(c))(\sigma E)
(\gamma_C^{-1}(b)\otimes1)\bigr)=(b\otimes1)E(1\otimes c),  \notag \\
&(\gamma_C\otimes\gamma_B)\bigl((\gamma_C^{-1}(b)\otimes1)(\sigma E)
(1\otimes\gamma_B^{-1}(c))\bigr)=(1\otimes c)E(b\otimes1).  \notag
\end{align}
\end{lem}

\begin{proof}
To verify the last equation, note that $(1\otimes c)E(b\otimes1)\in{\mathfrak M}_{\nu\otimes\mu}$. 
Applying $\nu\otimes\mu$, we have: 
\begin{align}
&(\nu\otimes\mu)\bigl((1\otimes c)E(b\otimes1)\bigr)=(\nu\otimes\mu)\bigl((1\otimes c)
EE(b\otimes1)\bigr)   \notag \\
&=(\nu\otimes\mu)\bigl([\gamma_C(c)\otimes1]E[1\otimes\gamma_B(b)]\bigr)  \notag \\
&=(\mu\otimes\nu)\bigl([1\otimes\gamma_C(c)](\sigma E)[\gamma_B(b)\otimes1]\bigr) \notag \\
&=(\mu\otimes\nu)\bigl([\sigma^{\mu}_{i}(\gamma_B(b))\otimes1](\sigma E)
[1\otimes\sigma^{\nu}_{-i}(\gamma_C(c))]\bigr) \notag
\end{align}
By Proposition~\ref{modularauto}, we know that $\sigma^{\mu}_{i}(\gamma_B(b))=(\gamma_C^{-1}
\circ\gamma_B^{-1})(\gamma_B(b))=\gamma_C^{-1}(b)$, and $\sigma^{\nu}_{-i}\bigl(\gamma_C(c)\bigr)
=(\gamma_B^{-1}\circ\gamma_C^{-1})\bigl(\gamma_C(c)\bigr)=\gamma_B^{-1}(c)$.  So we have:
$$
(\nu\otimes\mu)\bigl((1\otimes c)E(b\otimes1)\bigr)
=(\mu\otimes\nu)\bigl([\gamma_C^{-1}(b)\otimes1](\sigma E)
[1\otimes\gamma_B^{-1}(c)]\bigr).
$$
Since $\mu=\nu\circ\gamma_C$ and $\nu=\mu\circ\gamma_B$, we thus have:
$$
(\nu\otimes\mu)\bigl((1\otimes c)E(b\otimes1)\bigr)
=(\nu\circ\gamma_C\otimes\mu\circ\gamma_B)\bigl([\gamma_C^{-1}(b)\otimes1](\sigma E)
[1\otimes\gamma_B^{-1}(c)]\bigr).
$$
This result is true for any $b\in{\mathcal T}_{\nu}$, $c\in{\mathcal T}_{\mu}$, which are 
dense in $B$ and $C$, respectively.  Moreover, we know that the weights $\nu$, $\mu$ are 
faithful.  It follows that
$$
(1\otimes c)E(b\otimes1)=(\gamma_C\otimes\gamma_B)\bigl([\gamma_C^{-1}(b)\otimes1]
(\sigma E)[1\otimes\gamma_B^{-1}(c)]\bigr),
$$
proving the claim.  Other cases can be proved similarly.
\end{proof}

By Lemma~\ref{thelemma}, we can now prove the following result:

\begin{prop}\label{sigmaE}
We have: $(\gamma_C\otimes\gamma_B)(\sigma E)=E$ and $(\gamma_B\otimes\gamma_C)(E)
=\sigma E$.
\end{prop}

\begin{proof}
As written, we do not know whether $(\gamma_C\otimes\gamma_B)(\sigma E)$ is bounded. 
However, for $b\in{\mathcal T}_{\nu}$ (dense in $B$) and $c\in{\mathcal T}_{\mu}$ (dense in $C$), 
we know from Lemma~\ref{thelemma} that it can be made sense as follows:
\begin{align}
\bigl[(\gamma_C\otimes\gamma_B)(\sigma E)\bigr](b\otimes c)&:=(\gamma_C\otimes\gamma_B)
\bigl((\gamma_C^{-1}(b)\otimes\gamma_B^{-1}(c))(\sigma E)\bigr)=E(b\otimes c),  \notag \\
(b\otimes c)\bigl[(\gamma_C\otimes\gamma_B)(\sigma E)\bigr]&:=(\gamma_C\otimes\gamma_B)
\bigl((\sigma E)(\gamma_C^{-1}(b)\otimes\gamma_B^{-1}(c))\bigr)=(b\otimes c)E.
\notag
\end{align}
This means that $(\gamma_C\otimes\gamma_B)(\sigma E)$ coincides with $E\in M(B\otimes C)$ 
as a left and right multiplier map, on a dense subset of $B\otimes C$.  Since $E$ is bounded, 
this implies that $(\gamma_C\otimes\gamma_B)(\sigma E)$ can be canonically extended to 
a left and right multiplier map on all of $B\otimes C$, which would mean that 
$(\gamma_C\otimes\gamma_B)(\sigma E)\in M(B\otimes C)$, and that 
$(\gamma_C\otimes\gamma_B)(\sigma E)=E$. 
By taking the flip map, we also have: $(\gamma_B\otimes\gamma_C)(E)=\sigma E$.
\end{proof}

\begin{cor}
We have: $(R^{-1}\otimes R)(\sigma E)=E$ and $(R\otimes R^{-1})(E)=\sigma E$.
\end{cor}

\begin{proof}
From the proposition, we know: $\sigma E=(\gamma_B\otimes\gamma_C)(E)\in M(C\otimes B)$. 
Apply $R^{-1}\otimes R$ to both sides.  By Proposition~\ref{gammagamma'eqns}, we know 
$\gamma_B=R\circ\sigma^{\nu}_{i/2}$ and $\gamma_C=R^{-1}\circ\sigma^{\mu}_{-i/2}$.  Thus we have:
$$
(R^{-1}\otimes R)(\sigma E)=(R^{-1}\otimes R)(\gamma_B\otimes\gamma_C)(E)
=(\sigma^{\nu}_{i/2}\otimes\sigma^{\mu}_{-i/2})(E)=E.
$$
For the last equality, we used the result of Proposition~\ref{sigmasigmaE}.  

By taking the flip map, we also have: $(R\otimes R^{-1})(E)=\sigma E$.
\end{proof}

\section{Examples and special cases}

\subsection{Groupoids}

As indicated in the Introduction, the theory of weak multiplier Hopf algebras is motivated by 
attempts to generalize the notion of a groupoid.  Our theory is essentially a $C^*$-algebraic 
counterpart to the weak multiplier Hopf algebra theory.  Therefore, while it is true that 
this is not fully general due to the fact that the existence condition for the separability idempotent 
element is rather strong to be compatible with some topological aspects, it remains the case 
that a certain subclass of groupoids provide us with typical examples.  See below.

\bigskip

{\bf Example~4.A.}
Let $G$ be a discrete groupoid (equipped with the discrete topology).  Set the notations for the 
unit space $G^{(0)}$, the source map $s_G:G\to G^{(0)}$, and the target map $t_G:G\to G^{(0)}$, 
as in the Introduction.  Consider the commutative $C^*$-algebras $A=C_0(G)$ and $M(A)=C_b(G)$. 
Let $s,t:C_0(G^{(0)})\to M(A)$ be the pull-back maps corresponding to $s_G$ and $t_G$, respectively. 
That is, for $f\in C_0(G^{(0)})$, we have $s(f)\in M(A)=C_b(G)$ such that $s(f)(p)
=f\bigl(s_G(p)\bigr)$, and similarly for the map $t$.  Let $B$ and $C$ be the images under the 
maps $s$ and $t$, which are $C^*$-subalgebras of $M(A)$.  We have $M(B)$ and $M(C)$ 
contained in $M(A)$ as subalgebras.

Let $E\in M(A\otimes A)=C_b(G\times G)$ be such that 
$$
E(p,q):=\left\{\begin{matrix}1 & {\text { if $s_G(p)=t_G(q)$}} \\ 0 & {\text { otherwise}}
\end{matrix}\right.
$$
Since $E(p,q)=E\bigl(s_G(p),t_G(q)\bigr)$, we observe easily that $E\in M(B\otimes C)$. 

Consider next the counting measure on $G^{(0)}$, which naturally determines a weight on 
$C_0(G^{(0)})$.  Using the pull-back maps, we can define weights $\nu$ on $B$ and $\mu$ 
on $C$.  They are faithful, tracial weights (so KMS).

To describe the $R$ map, consider $s(f)=f\circ s_G\in B$, where $f\in C_0(G^{(0)})$. 
Then we have $R(f\circ s_G)\in C$, given by $R(f\circ s_G)(p)=(f\circ t_G)(p^{-1})$. 
It is a well-defined map, because whenever $s_G(p)=s_G(q)\in G^{(0)}$, we have 
$t_G(p^{-1})=t_G(q^{-1})$.  It is clear that it is an (anti)-isomorphism.

It is not difficult to see that $(E,B,\nu)$ forms a separability triple in the sense of 
Definition~\ref{separabilitytriple}, and therefore $E$ is a separability idempotent.

\bigskip

{\bf Example~4.B.}
Consider again a discrete groupoid $G$.  Denote by $K(G)$ the space of all complex 
functions on $G$ having finite support.  In particular, for $p\in G$, define $\lambda_p\in K(G)$ 
by $\lambda_p(x):=\delta_{p,x}=\left\{\begin{matrix} 1 & {\text { if $x=p$}} \\ 
0 & {\text { otherwise}}\end{matrix}\right.$.  Note that any function $f\in K(G)$ can be 
expressed in the form $f=\sum_{p\in G}f(p)\lambda_p$.  We can give $K(G)$ a ${}^*$-algebra 
structure by letting $\lambda_p\lambda_q=\lambda_{pq}$, valid only when $pq$ is defined 
and 0 otherwise; while at the same time letting $\lambda_p^*=\lambda_{p^{-1}}$.  This is 
none other than the convolution ${}^*$-algebra structure.

Similar to the group case, one can consider the (left) regular representation of $K(G)$ on 
$l^2(G)$, given by the counting measure on $G$.  The regular representation extends 
$K(G)$ to a $C^*$-algebra $A=C^*_r(G)$.  Unless $G^{(0)}$ is a finite set, the algebra $A$ 
is in general non-unital.  However, at the level of the multiplier algebra, the unit element is 
$1=\sum\lambda_e$, where the sum is taken over all $e\in G^{(0)}$.  We will skip the details 
(left Haar system, and the like) and refer to the standard textbooks (\cite{Renbook}, \cite{Patbook}), 
as the issue at hand is more about the separability idempotents than about groupoid algebras. 
More discussion on the (quantum) groupoid aspects will be given in the authors' upcoming papers 
\cite{BJKVD_qgroupoid1}, \cite{BJKVD_qgroupoid2}. 

Note that if $e\in G^{(0)}$, we have $s_G(e)=t_G(e)=e$.  Because of this, the base algebras 
$B$ and $C$, which are generated by $\{\lambda_{s_G(p)}:p\in G\}$ and $\{\lambda_{t_G(p)}:p\in G\}$ 
respectively, will turn out to be isomorphic to $C_0(G^{(0)})$.  The counting measure on $G^{(0)}$ 
would provide the weights $\nu$ and $\mu$. Being faithful and tracial, they are KMS weights. 
Let $R:B\to C$ be given by $\lambda_{s_G(p)}\mapsto\lambda_{t_G(p^{-1})}$.  It is a well-defined 
map, which is clearly an (anti)-isomorphism.

Now consider $E:=\sum\lambda_e\otimes\lambda_e$, where the sum is taken over all $e\in G^{(0)}$. 
It is not difficult to show that it becomes a separability idempotent.

\begin{rem}
(1). Since $G$ is discrete, the topology does not play much significant role here.  At the ${}$*-algebra 
level, the two examples above can be described in the framework of weak multiplier Hopf algebras 
\cite{VDWangwha0}, \cite{VDWangwha1}.    For a general weak multiplier Hopf algebra, the discussion 
about $E$ being a separability idempotent in terms of its ``distinguished linear functionals'' ($\mu$ 
and $\nu$ in our case) can be found in section 4 of \cite{VDsepid}.

(2). There also exists the notion of a {\em measured quantum groupoid\/} by Lesieur and Enock, 
in the von Neumann algebra framework \cite{LesSMF}, \cite{EnSMF}.  It generalizes the notion of 
a groupoid, so in some special cases (for instance, when the base algebra is finite-dimensional), 
they provide a separability idempotent.  But we will postpone the discussion on all these cases 
related to ``quantum groupoids'' to our future papers \cite{BJKVD_qgroupoid1}, 
\cite{BJKVD_qgroupoid2}.
\end{rem}

\subsection{An example coming from an action groupoid}

Consider the following specific example: The group $(\mathbb{Z},+)$ acts on $\overline{\mathbb{Z}}
:=\mathbb{Z}\cup\{\infty\}$ by translation and leaving $\infty$ fixed, and let $\overline{\mathbb{Z}}
\times\mathbb{Z}$ be the corresponding action groupoid (or, transformation group groupoid). 
It can be considered as a locally compact groupoid, together with the discrete topology on 
$\mathbb{Z}$, and $\overline{\mathbb{Z}}$ being the compactification to $\infty$.  Then consider 
$G:=(\overline{\mathbb{Z}}\times\mathbb{Z})|_{\overline{\mathbb{N}}}$, the restriction of the action 
groupoid to $\overline{\mathbb{N}}\subset\overline{\mathbb{Z}}$.

$G$ is also a groupoid (the ``Cuntz groupoid'', see \cite{Renbook}, \cite{Patbook}, \cite{BCST}), 
where $G^{(0)}=\overline{\mathbb{N}}$ and equipped with the target map $t:(m,p)\mapsto m$ and the 
source map $s:(m,p)\mapsto m+p$.  We have: $(m,p)\cdot(n,q)=\delta_{m+p,n}(m,p+q)$.  It is a 
locally compact groupoid, whose topology is inherited from that of the action groupoid $\bar
{\mathbb{Z}}\times\mathbb{Z}$.  The (left) Haar system is given by the discrete measure on 
each $t$-fiber. 

We can certainly consider, as in Example~4.B above, a natural separability idempotent 
$E_0:=\sum_{m\in\overline{\mathbb{N}}}\lambda_{m,0}\otimes\lambda_{m,0}$, which is 
characterized by $B_0\cong C_0(G^{(0)})$ and $\nu_0$, the counting measure on $B_0$, 
namely, $\nu_0(\lambda_{m,0})=1$, $\forall m\in\overline{\mathbb{N}}$.  However, for this next 
example, we instead wish to consider all of $\mathbb{C}G$ as the base algebra for our (soon 
to be constructed) separability idempotent.

Consider $B:=C^*(G)=\mathbb{C}G$.  As in Example~4.B, it can be realized as a completion 
of the convolution algebra $C_c(G)$ given by the convolution product $\lambda_{m,p}\lambda_{n,q}
=\delta_{m+p,n}\lambda_{m,p+q}$ and the involution $\lambda_{m,p}^*=\lambda_{m+p,-p}$. 
As before, we are regarding the generators $\lambda_{m,p}$ as functions on $G$, given by 
$\lambda_{m,p}(n,q)=\delta_{m,n}\delta_{p,q}$.  It is known that $B$ is isomorphic to 
the $C^*$-algebra generated by the unilateral shift operator on $l^2(\mathbb{N})$, but 
we do not need this fact for our purposes (see \cite{Renbook}).

Fix a positive constant $\hbar>0$.  Define:
\begin{equation}\label{(Eexample2)}
E:=\sum_{m,p\in\mathbb{N}}e^{-(p+m)\hbar/2}(1-e^{-\hbar})\lambda_{m,p-m}\otimes\lambda_{m,p-m}.
\end{equation}
Actually, the sum is taken over all $(m,p-m)\in G$, so $t(m,p-m)=m\in\overline{\mathbb{N}}$ and 
$s(m,p-m)=p\in\overline{\mathbb{N}}$, but with the convention that $e^{-\infty}=0$.  In what follows, 
we will show that there exists a suitable weight $\nu_{\hbar}$ on $B$ so that $(E,B,\nu_{\hbar})$ 
forms a separability triple.

\begin{prop}
$E$ is self-adjoint and idempotent.
\end{prop}

\begin{proof}
(1). By definition, $\lambda_{m,p-m}^*=\lambda_{m+(p-m),-(p-m)}=\lambda_{p,m-p}$.  So we have 
$$
E^*=\sum_{m,p\in\mathbb{N}}e^{-(p+m)\hbar/2}(1-e^{-\hbar})\lambda_{p,m-p}\otimes\lambda_{p,m-p}\,=\,E.
$$
(2). Since $\lambda_{m,p-m}\lambda_{n,q-n}=\delta_{m+(p-m),n}\lambda_{m,(p-m)+(q-n)}
=\delta_{p,n}\lambda_{m,q-m}$, we have:
\begin{align}
E^2&=\sum_{m,p,n,q\in\mathbb{N}}e^{-(p+m)\hbar/2}e^{-(q+n)\hbar/2}(1-e^{-\hbar})^2
\lambda_{m,p-m}\lambda_{n,q-n}\otimes\lambda_{m,p-m}\lambda_{n,q-n}  \notag \\
&=\sum_{m,p,q\in\mathbb{N}}e^{-p\hbar}e^{-(q+m)\hbar/2}(1-e^{-\hbar})^2
\lambda_{m,q-m}\otimes\lambda_{m,q-m}  \notag \\
&=\sum_{m,q\in\mathbb{N}}e^{-(q+m)\hbar/2}(1-e^{-\hbar})
\lambda_{m,q-m}\otimes\lambda_{m,q-m}\,=\,E.  \notag
\end{align}
The second equality is using the fact that $n=p$, and the third equality is because $\sum_{p\in\mathbb{N}}
e^{-p\hbar}(1-e^{-\hbar})=1$.
\end{proof}

Let $C=C^*(G)$, which will be considered as an opposite $C^*$-algebra of $B$, via the map $R:B\to C$, 
defined on the generators by $R(\lambda_{n,q})=\lambda_{n+q,-q}$.  Coming from the involution, it is 
clear that $R$ is indeed a ${}^*$-anti-isomorphism.

Next, define the weight $\nu_{\hbar}$ on $B$, whose values on the generators are 
\begin{align}\label{(nuExample2)}
&\nu_{\hbar}(\lambda_{m,0})=e^{m\hbar}(1-e^{-\hbar})^{-1}, {\text { for $m=0,1,2,\dots$}}; \notag \\
&\nu_{\hbar}(\lambda_{m,p})=0, {\text { when $p\ne0$.}}
\end{align}
So for $f\in C_c(G)$, we have: $\nu_{\hbar}(f)=\sum_{m=0}^{\infty}e^{m\hbar}(1-e^{-\hbar})^{-1}f(m,0)$. 
With the values known for the generators, it is clear that $\nu_{\hbar}$ is a proper weight on $B$.  It is 
in fact a KMS weight, together with the one-parameter group $(\sigma^{\hbar}_t)_{t\in\mathbb{R}}$, 
defined by $\sigma^{\hbar}_t(\lambda_{n,q}):=e^{-itq\hbar}\lambda_{n,q}$.  See proposition below.

\begin{prop}\label{proposition_example2}
Let $E$ be as in Equation~\eqref{(Eexample2)}, and consider the weight $\nu_{\hbar}$ on $B$ 
as in Equation~\eqref{(nuExample2)}.  Then: 
\begin{enumerate}
  \item $\nu_{\hbar}$, together with the one-parameter group $(\sigma^{\hbar}_t)_{t\in\mathbb{R}}$, 
        is a KMS weight on $B$.
  \item For each $\lambda_{n,q}$, we have: 
$E(\lambda_{n,q}\otimes1)=E(1\otimes e^{q\hbar/2}\lambda_{n+q,-q})$.
  \item $(\nu_{\hbar}\otimes\operatorname{id})(E)=1$.
  \item $(\nu_{\hbar}\otimes\operatorname{id})\bigl(E(\lambda_{n,q}\otimes1)\bigr)
        =R\bigl(\sigma^{\hbar}_{i/2}(\lambda_{n,q})\bigr)$.
\end{enumerate}
\end{prop}

\begin{proof}
(1). To see if $\nu_{\hbar}$ is faithful, let $\nu_{\hbar}(f^*f)=0$, where $f=\sum_{m,p}f(m,p)\lambda_{m,p}$. 
Since 
$\lambda_{n,q}^*\lambda_{m,p}=\lambda_{n+q,-q}\lambda_{m,p}=\delta_{n,m}\lambda_{n+q,p-q}$, 
we have: 
$$
f^*f=\sum_{m,p,n,q}\overline{f(n,q)}f(m,p)\lambda_{n,q}^*\lambda_{m,p}
=\sum_{m,p,q}\overline{f(m,q)}f(m,p)\lambda_{m+q,p-q}.
$$
Note also that $\nu_{\hbar}(\lambda_{m+q,p-q})=0$ unless $p=q$.  It follows that
$$
\nu_{\hbar}(f^*f)=\sum_{m,p}\overline{f(m,p)}f(m,p)\nu_{\hbar}(\lambda_{m+p,0}).
$$
But we always have $\nu_{\hbar}(\lambda_{m+p,0})>0$, while $\overline{f(m,p)}f(m,p)\ge0$.  This 
means that $\nu_{\hbar}(f^*f)=0$ only if $f(m,p)\equiv0$, showing that $\nu_{\hbar}$ is faithful.

Meanwhile, with the $(\sigma^{\hbar}_t)$ given above, we have:
$$
\nu_{\hbar}\bigl(\sigma^{\hbar}_t(\lambda_{n,q})\bigr)=\nu_{\hbar}(e^{-itq\hbar}\lambda_{n,q})
=\left\{\begin{matrix}0 & {\text {if $q\ne0$}} \\ \nu_{\hbar}(\lambda_{n,0}) & {\text {if $q=0$}}
\end{matrix}\right.
$$
So for any $t\in\mathbb{R}$, we see that $\nu_{\hbar}\bigl(\sigma^{\hbar}_t(\lambda_{n,q})\bigr)
=\nu_{\hbar}(\lambda_{n,q})$, $\forall n,q$, showing the invariance of $\nu_{\hbar}$ under 
the $(\sigma^{\hbar}_t)$.  In addition, 
\begin{align}
\nu_{\hbar}\bigl(\sigma^{\hbar}_{i/2}(\lambda_{n,q})\sigma^{\hbar}_{i/2}(\lambda_{n,q})^*\bigr)
&=\nu_{\hbar}\bigl((e^{q\hbar/2}\lambda_{n,q})(e^{q\hbar/2}\lambda_{n+q,-q})\bigr)
=e^{q\hbar}\nu_{\hbar}(\lambda_{n,0})  \notag \\
&=e^{(q+n)\hbar}(1-e^{-\hbar})^{-1},
\notag
\end{align}
which coincides with $\nu_{\hbar}(\lambda_{n,q}^*\lambda_{n,q})=\nu_{\hbar}(\lambda_{n+q,-q}\lambda_{n,q})
=\nu_{\hbar}(\lambda_{n+q,0})$.  Therefore, by Definition~\ref{KMSweight}, we see that $\nu_{\hbar}$ 
is indeed a KMS weight.

(2). Let $E$ be as in equation~\eqref{(Eexample2)}.  Then: 
\begin{align}
E(\lambda_{n,q}\otimes1)&=\sum_{m,p}e^{-(p+m)\hbar/2}(1-e^{-\hbar})\lambda_{m,p-m}\lambda_{n,q}
\otimes\lambda_{m,p-m}  \notag \\
&=\sum_m e^{-(n+m)\hbar/2}(1-e^{-\hbar})\lambda_{m,n-m+q}\otimes\lambda_{m,n-m}
\notag
\end{align}
while
\begin{align}
E(1\otimes e^{q\hbar/2}\lambda_{n+q,-q})&=\sum_{m,p}e^{-(p+m)\hbar/2}(1-e^{-\hbar})
e^{q\hbar/2}\lambda_{m,p-m}\otimes\lambda_{m,p-m}\lambda_{n+q,-q} \notag \\
&=\sum_{m}e^{-(n+q+m)\hbar/2}(1-e^{-\hbar})e^{q\hbar/2}\lambda_{m,n+q-m}\otimes\lambda_{m,n-m} 
\notag
\end{align}
Comparing, we see that $E(\lambda_{n,q}\otimes1)=E(1\otimes e^{q\hbar/2}\lambda_{n+q,-q})$.

(3). With the definition of $\nu_{\hbar}$ given in Equation~\eqref{(nuExample2)}, we have:
\begin{align}\label{(gammamap_example2)}
&(\nu_{\hbar}\otimes\operatorname{id})\bigl[E(\lambda_{n,q}\otimes1)\bigr] \notag \\
&=(\nu_{\hbar}\otimes\operatorname{id})\bigl[\sum_m e^{-(n+m)\hbar/2}(1-e^{-\hbar})
\lambda_{m,n-m+q}\otimes\lambda_{m,n-m}\bigr] 
\notag \\
&=\sum_m e^{-(n+m)\hbar/2}(1-e^{-\hbar})e^{m\hbar}(1-e^{-\hbar})^{-1}\delta_{n-m+q,0}
\lambda_{m,n-m}=e^{q\hbar/2}\lambda_{n+q,-q}.
\end{align}
By (2), this can be realized as $(\nu_{\hbar}\otimes\operatorname{id})\bigl[E(1\otimes e^{q\hbar/2}
\lambda_{n+q,-q})\bigr]=e^{q\hbar/2}\lambda_{n+q,-q}$, which is equivalent to $(\nu_{\hbar}
\otimes\operatorname{id})\bigl[E(1\otimes\lambda_{m,p})\bigr]=\lambda_{m,p}$, true for any 
generator $\lambda_{m,p}$.  This shows that $(\nu_{\hbar}\otimes\operatorname{id})(E)=1$.

(4). By definition, we know $\sigma^{\hbar}_{i/2}(\lambda_{n,q})=e^{q\hbar/2}\lambda_{n,q}$.  
So $R\bigl(\sigma^{\hbar}_{i/2}(\lambda_{n,q})\bigr)=e^{q\hbar/2}\lambda_{n+q,-q}$.  Then 
the result obtained in equation~\eqref{(gammamap_example2)} is none other than saying 
$(\nu_{\hbar}\otimes\operatorname{id})\bigl(E(\lambda_{n,q}\otimes1)\bigr)=e^{q\hbar/2}\lambda_{n+q,-q}
=R\bigl(\sigma^{\hbar}_{i/2}(\lambda_{n,q})\bigr)$.
\end{proof}
 
By Definition~\ref{separabilitytriple}, we conclude from Proposition~\ref{proposition_example2}
that $(E,B,\nu_{\hbar})$ forms a separability triple, and therefore, $E$ is a separability idempotent. 
We may have an occasion in the future to explore this example further.

\subsection{A case when $B={\mathcal B}_0({\mathcal H})$} 

We wish to study an example that generalizes the example given in section~4 of \cite{VDsepid}. 
Consider a Hilbert space ${\mathcal H}$, possibly infinite-dimensional.  Fix an orthonormal 
basis $(\xi_j)_{j\in J}$ for ${\mathcal H}$, where $J$ is an index set.  For any 
$\mathbf{v}\in{\mathcal H}$, we may write $\mathbf{v}=\sum_{j\in J}v_j\xi_j$, where 
$v_j=\langle\mathbf{v},\xi_j\rangle$. Consider also $\operatorname{Tr}(\,\cdot\,)$, the canonical 
trace on ${\mathcal B}({\mathcal H})$. 

Define $E_0:=\sum_{i,j\in J}e_{ij}\otimes e_{ij}$, where the $e_{ij}\in{\mathcal B}({\mathcal H})$ 
are the matrix units, given by 
$$
e_{ij}(\mathbf{v}):=\langle\mathbf{v},\xi_j\rangle\xi_i=v_j\xi_i,\quad{\text { for $\mathbf{v}
\in{\mathcal H}$}}.
$$

\begin{lem}\label{lemLA}
With the notations above, we have the following results:
\begin{enumerate}
  \item For $a\in{\mathcal B}({\mathcal H})$ and $\mathbf{v}\in{\mathcal H}$, we have, for $i\in J$,  
  $(a\mathbf{v})_i=\sum_{j\in J}a_{ij}v_j$, where $a_{ij}=\langle a\xi_j,\xi_i\rangle$.
  \item For $a\in{\mathcal B}({\mathcal H})$, we can consider $a^T\in{\mathcal B}({\mathcal H})$, 
  the ``transpose'' of $a$, defined by $a^T_{ij}=a_{ji}=\langle a\xi_i,\xi_j\rangle$, for $i,j\in J$. 
  We will have: $\operatorname{Tr}(a)=\sum_{j\in J}\langle a\xi_j,\xi_j\rangle=\sum_{j\in J}a_{jj}
  =\operatorname{Tr}(a^T)$.
  \item For the (unbounded) operator $E_0$ defined above, we have: 
  $$E_0(\mathbf{v}\otimes\mathbf{w})=\sum_{i,j\in J}v_jw_j\xi_i\otimes\xi_i,\qquad
  {\text { for $\mathbf{v},\mathbf{w}\in{\mathcal H}$.}}$$
\end{enumerate}
\end{lem}

\begin{proof}
Straightforward: Basic linear algebra.
\end{proof}

\begin{prop}\label{E0}
Let $a\in{\mathcal B}({\mathcal H})$.  Then we have:
$$
E_0(a\otimes1)=E_0(1\otimes a^T).
$$
\end{prop}

\begin{proof}
For $\mathbf{v},\mathbf{w}\in{\mathcal H}$, by using Lemma~\ref{lemLA}, we have:
\begin{align}
E_0(a\otimes1)(\mathbf{v}\otimes\mathbf{w})&=E_0(a\mathbf{v}\otimes\mathbf{w})  \notag \\
&=\sum_{i,j\in J}(a\mathbf{v})_jw_j\xi_i\otimes\xi_i
=\sum_{i,j,k\in J}(a_{jk}v_k)w_j\xi_i\otimes\xi_i  \notag \\
&=\sum_{i,j,k\in J}v_k(a^T_{kj}w_j)\xi_i\otimes\xi_i
=\sum_{i,k\in J}v_k(a^T\mathbf{w})_k\xi_i\otimes\xi_i \notag \\
&=E_0(\mathbf{v}\otimes a^T\mathbf{w}).  \notag
\end{align}
This is true for all $\mathbf{v},\mathbf{w}\in{\mathcal H}$.  So $E_0(a\otimes1)=E_0(1\otimes a^T)$.
\end{proof}

Consider $r\in{\mathcal B}({\mathcal H})$, which is invertible and $\operatorname{Tr}(r^*r)=1$. 
So $r$ is a compact operator of Hilbert--Schmidt type.  This means that $r^{-1}$ is unbounded 
in general, unless ${\mathcal H}$ is finite-dimensional.  Next, define a new operator $E$, by
\begin{equation}\label{(E)}
E:=(r\otimes1)E_0(r^*\otimes1).  
\end{equation}
To learn more about the property of $E$, we will first consider a lemma.

\begin{lem}\label{leme_ij}
Let $e_{ij}\in{\mathcal B}({\mathcal H})$ be as above.  Then we have:
\begin{enumerate}
  \item $e_{ij}e_{kl}=\delta_{jk}e_{il}$.
  \item $e_{ij}^*=e_{ji}$.
  \item For $a\in{\mathcal B}({\mathcal H})$, we have: $\sum_{j\in J}e_{ij}ae_{jl}
  =\operatorname{Tr}(a)e_{il}$.
\end{enumerate}
\end{lem}

\begin{proof}
(1). Let $\mathbf{v}\in{\mathcal H}$ be arbitrary.  We have: 
$$
e_{ij}e_{kl}(\mathbf{v})=e_{ij}\bigl(e_{kl}(\mathbf{v})\bigr)=e_{ij}(v_l\xi_k)
=v_l\delta_{jk}\xi_i=\delta_{jk}e_{il}(\mathbf{v}).
$$ 

(2). Let $\mathbf{v},\mathbf{w}\in{\mathcal H}$ be arbitrary. We have:
$$
\bigl\langle e_{ij}(\mathbf{v}),\mathbf{w}\bigr\rangle=\langle v_j\xi_i,\mathbf{w}\rangle=v_jw_i
=\langle\mathbf{v},w_i\xi_j\rangle=\bigl\langle\mathbf{v},e_{ji}(\mathbf{w})\bigr\rangle.
$$

(3). Let $\mathbf{v}\in{\mathcal H}$ be arbitrary.  We have: 
\begin{align}
\sum_{j\in J}e_{ij}ae_{jl}(\mathbf{v})&=\sum_{j\in J}e_{ij}a(v_l\xi_j)=\sum_{j\in J}v_le_{ij}
\left(\sum_{k\in J}a_{kj}\xi_k\right)=\sum_{j,k\in J}v_la_{kj}e_{ij}(\xi_k)  \notag \\
&=\sum_{j,k\in J}v_la_{kj}\delta_{jk}\xi_i=\sum_{j\in J}a_{jj}v_l\xi_i
=\operatorname{Tr}(a)e_{il}(\mathbf{v}).   \notag
\end{align}
\end{proof}

By using the result of Lemma~\ref{leme_ij}, we obtain the following proposition:

\begin{prop}\label{Eprojection}
Let $E$ be defined as in \eqref{(E)} above.  Then $E\in{\mathcal B}({\mathcal H}\otimes{\mathcal H})$, 
and we have: $E^*=E=E^2$.
\end{prop}

\begin{proof}
By Lemma~\ref{leme_ij}\,(2), we observe that $E_0^*=\left(\sum_{i,j\in J}e_{ij}\otimes e_{ij}\right)^*
=\sum_{i,j\in J}e_{ji}\otimes e_{ji}=E_0$.  It follows that $E^*=\bigl((r\otimes1)E_0(r^*\otimes1)\bigr)^*
=E$.

By Lemma~\ref{leme_ij}\,(1)\,\&\,(3), we see that 
\begin{align}
E_0(r^*r\otimes1)E_0&=\sum_{i,j,k,l\in J}(e_{ij}\otimes e_{ij})(r^*r\otimes1)(e_{kl}\otimes e_{kl})  \notag \\
&=\sum_{i,j,k,l\in J}e_{ij}r^*re_{kl}\otimes\delta_{jk}e_{il} =\sum_{i,j,l\in J}e_{ij}r^*re_{jl}\otimes e_{il}   \notag \\
&=\sum_{i,l\in J}\operatorname{Tr}(r^*r)e_{il}\otimes e_{il}
=\operatorname{Tr}(r^*r)E_0=E_0.  \notag
\end{align}
In the last line, we used the assumption that $\operatorname{Tr}(r^*r)=1$.  It follows that 
$E^2=(r\otimes1)E_0(r^*r\otimes1)E_0(r^*\otimes1)=(r\otimes1)E_0(r^*\otimes1)=E$.

We see that $E$ is a self-adjoint idempotent, so bounded projection in ${\mathcal B}({\mathcal H}
\otimes{\mathcal H})$.
\end{proof}

Next, define $B:={\mathcal B}_0({\mathcal H})$, the $C^*$-algebra of compact operators on ${\mathcal H}$.
Let us consider a suitable weight on $B$:

\begin{prop}\label{weightTrp}
Consider the map $\nu:=\operatorname{Tr}(r^{-1}\,\cdot\,{r^*}^{-1})=\operatorname{Tr}\bigl(p\,\cdot)$, 
where $p={r^*}^{-1}r^{-1}=(rr^*)^{-1}$.  It is a KMS weight on the $C^*$-algebra $B$. 
The associated norm-continuous automorphism group, the ``modular automorphism group'' for $\nu$, 
is $(\sigma^{\nu}_t)_{t\in\mathbb{R}}$, defined by $\sigma^{\nu}_t(b)=p^{it}bp^{-it}$.
\end{prop}

\begin{proof}
From the properties of $\operatorname{Tr}(\,\cdot\,)$, which is itself a proper weight, it is easy to 
see that $\nu$ is a lower semi-continuous weight and is densely-defined.  It is also faithful, because 
if $\nu(b^*b)=\operatorname{Tr}(r^{-1}b^*b{r^*}^{-1})=0$,  it necessarily means $b{r^*}^{-1}$=0, 
so $b=0$.

With the definition of the $(\sigma^{\nu}_t)$ given above and for $b\in{\mathfrak M}_{\nu}$, 
observe that $\nu\bigl(\sigma^{\nu}_t(b)\bigr)=\operatorname{Tr}(pp^{it}bp^{-it})=\operatorname{Tr}
(p^{-it}pp^{it}b)=\operatorname{Tr}(pb)=\nu(b)$, $t\in\mathbb{R}$.  This shows the invariance 
of $\nu$ under the $(\sigma^{\nu}_t)$.  While, for $b\in{\mathcal D}(\sigma^{\nu}_{i/2})$, we have:
\begin{align}
\nu\bigl(\sigma^{\nu}_{i/2}(b)\sigma^{\nu}_{i/2}(b)^*\bigr)
&=\operatorname{Tr}\bigl(p(p^{-\frac12}bp^{\frac12})(p^{-\frac12}bp^{\frac12})^*\bigr)
=\operatorname{Tr}(p^{\frac12}bpb^*p^{-\frac12})  \notag \\
&=\operatorname{Tr}(pb^*p^{-\frac12}p^{\frac12}b)=\operatorname{Tr}(pb^*b)
=\nu(b^*b).  \notag
\end{align}
By the conditions given in Definition~\ref{KMSweight}, we see that $\nu$ is a KMS weight.
\end{proof}

Let $C={\mathcal B}_0({\mathcal H})$.  Note that $M(B)=M(C)={\mathcal B}({\mathcal H})$. 
Meanwhile, by polar decomposition, write $r^*=u|r^*|$, where $|r^*|=(rr^*)^{\frac12}=p^{-\frac12}$. 
Now let $R:B\to C$ be defined by $R(b):=(ubu^*)^T$.  It is a ${}^*$-anti-isomorphism, with 
$R^{-1}(c)=u^*c^Tu$.  It will soon be made clear why we choose $R$ in this way.

With the definition given in Equation~\eqref{(E)}, we showed in Proposition~\ref{Eprojection} 
that $E$ is self-adjoint and idempotent such that $E\in{\mathcal B}({\mathcal H}\otimes{\mathcal H})
=M(B\otimes C)$.  In the next proposition, we see that $(E,B,\nu)$ indeed forms a separability triple.

\begin{prop}\label{Exseptriple}
Consider the triple $(E,B,\nu)$ as given above.  Then
\begin{enumerate}
  \item $(\nu\otimes\operatorname{id})(E)=1_{{\mathcal B}({\mathcal H})}$, where 
  $1_{{\mathcal B}({\mathcal H})}$ is the identity in $M(B)={\mathcal B}({\mathcal H})$.
  \item For $b\in{\mathcal D}(\sigma^{\nu}_{i/2})$, we have: 
  $(\nu\otimes\operatorname{id})\bigl(E(b\otimes1)\bigr)=R\bigl(\sigma^{\nu}_{i/2}(b)\bigr)$.
\end{enumerate}
By Definition \ref{separabilitytriple}, we see that $(E,B,\nu)$ forms a separability triple, 
and therefore, $E$ is a separability idempotent.
\end{prop}

\begin{proof}
(1). To be rigorous, as noted in the Remark following Definition~\ref{separabilitytriple}, we should 
understand (1) as saying that $(\operatorname{id}\otimes\omega)(E)\in{\mathfrak M}_{\nu}$ 
for $\omega\in C^*$ and that $\nu\bigl((\operatorname{id}\otimes\omega)(E)\bigr)=\omega(1)$. 
As $C={\mathcal B}_0({\mathcal H})$, any $\omega\in C^*$ may be expressed in the form 
$\omega=\operatorname{Tr}(\,\cdot\,a)$, where $a$ is a trace-class operator.  Then, 
since $(e_{ij}a)_{mn}=\delta_{m,i}a_{jn}$, we have: 
\begin{equation}\label{(TrE0)}
(\operatorname{id}\otimes\omega)(E_0)
=(\operatorname{id}\otimes\operatorname{Tr})\bigl(E_0(1\otimes a)\bigr)
=\sum_{i,j,m\in J}e_{ij}(\delta_{m,i}a_{jm})=\sum_{i,j\in J}e_{ij}a_{ji}=a^T.
\end{equation}
With $p={r^*}^{-1}r^{-1}$, it is obvious from Equation~\eqref{(TrE0)} that 
$$
\operatorname{Tr}(a^T)=\operatorname{Tr}(r^*pra^T)=\operatorname{Tr}(pra^Tr^*)
=\operatorname{Tr}\bigl(pr(\operatorname{id}\otimes\omega)(E_0)r^*\bigr).
$$
Since we can write $r(\operatorname{id}\otimes\omega)(E_0)r^*
=(\operatorname{id}\otimes\omega)\bigl((r\otimes1)E_0(r^*\otimes1)\bigr)
=(\operatorname{id}\otimes\omega)(E)$, and since $\nu=\operatorname{Tr}(p\,\cdot\,)$, 
we can now see that $(\operatorname{id}\otimes\omega)(E)\in{\mathfrak M}_{\nu}$.  
We also have:
$$
\nu\bigl((\operatorname{id}\otimes\omega)(E)\bigr)
=\operatorname{Tr}\bigl(pr(\operatorname{id}\otimes\omega)(E_0)r^*\bigr)
=\operatorname{Tr}(a^T)=\operatorname{Tr}(a)=\omega(1),
$$
proving the claim.

(2). Recall the polar decomposition $r^*=u|r^*|=up^{-\frac12}$.  For $b\in B$, note that 
$r^*b{r^*}^{-1}=u|r^*|b|r^*|^{-1}u^*=up^{-\frac12}bp^{\frac12}u^*$.  We see that $r^*b{r^*}^{-1}$ 
is bounded if and only if $p^{-\frac12}bp^{\frac12}=\sigma^{\nu}_{i/2}(b)$ is bounded.
In particular, if $b\in{\mathcal D}(\sigma^{\nu}_{i/2})$, we know that $r^*b{r^*}^{-1}$ is bounded. 

Consider an arbitrary trace-class operator $a$, and consider $\omega=\operatorname{Tr}(\,\cdot\,a)$.
For the time being, write $s=(r^*b{r^*}^{-1})^T$, which is also bounded.  As the algebra of 
trace-class operators forms an ideal in ${\mathcal B}({\mathcal H})$, it follows that $sa$ is also 
trace-class.  Then as in Equation~\eqref{(TrE0)}, we will have:
$$
(\operatorname{id}\otimes\omega)\bigl(E_0(1\otimes s)\bigr)
=(\operatorname{id}\otimes\operatorname{Tr})\bigl(E_0(1\otimes sa)\bigr)
=(sa)^T.
$$
By Proposition~\ref{E0}, this can be re-written as:
$(\operatorname{id}\otimes\omega)\bigl(E_0(s^T\otimes1)\bigr)=(sa)^T$.
With $p={r^*}^{-1}r^{-1}$ and $s=(r^*b{r^*}^{-1})^T$, it follows that
\begin{align}
&\operatorname{Tr}((sa)^T)=\operatorname{Tr}(r^*pr(sa)^T)=\operatorname{Tr}(pr(sa)^Tr^*)  \notag \\
&=\operatorname{Tr}\bigl(pr(\operatorname{id}\otimes\omega)(E_0(s^T\otimes1))r^*\bigr)
=\operatorname{Tr}\bigl(pr(\operatorname{id}\otimes\omega)(E_0(r^*b{r^*}^{-1}\otimes1))r^*\bigr)   \notag \\
&=\operatorname{Tr}\bigl(p(\operatorname{id}\otimes\omega)(E(b\otimes1))\bigr).
\notag
\end{align}
Therefore, we can conclude that $(\operatorname{id}\otimes\omega)\bigl(E(b\otimes1)\bigr)\in{\mathfrak M}_{\nu}$, 
and that 
$$
\nu\bigl((\operatorname{id}\otimes\omega)(E(b\otimes1))\bigr)=\operatorname{Tr}((sa)^T)
=\operatorname{Tr}(sa)=\omega(s)=\omega\bigl((r^*b{r^*}^{-1})^T\bigr).
$$
This is true for any $\omega\in C^*$. In other words, we have:
$$
(\nu\otimes\operatorname{id})\bigl(E(b\otimes1)\bigr)
=(r^*b{r^*}^{-1})^T=\bigl(u\sigma^{\nu}_{i/2}(b)u^*\bigr)^T
=R\bigl(\sigma^{\nu}_{i/2}(b)\bigr).
$$
\end{proof}

Since $E\in M(B\otimes C)$ is a separability triple, all the results of Section~2 and Section~3 
will hold.  Here we gather some of them, which may be useful in the future.

Consider $\mu:=\nu\circ R^{-1}$.  By general theory (see Section 3), we know that it is a KMS 
weight on $C={\mathcal B}_0({\mathcal H})$, such that $(\operatorname{id}\otimes\mu)(E)
=1_{{\mathcal B}({\mathcal H})}$.  The following proposition gathers its properties.

\begin{prop}
With the definition above, we have: $\mu=\operatorname{Tr}(q\,\cdot\,)$, where $q=\bigl((r^*r)^{-1}\bigr)^T$. 
The modular automorphism group for $\mu$ is $(\sigma^{\mu}_t)_{t\in\mathbb{R}}$, given by
$$
\sigma^{\mu}_t(c)=(R\circ\sigma^{\nu}_{-t}\circ R^{-1})(c)=q^{it}cq^{-it},\quad c\in C.
$$
\end{prop}

\begin{proof}
(1). If $c\in{\mathfrak M}_{\mu}$, we will have:
$$
\mu(c)=\nu\bigl(R^{-1}(c)\bigr)=\nu(u^*c^Tu)=\operatorname{Tr}(pu^*c^Tu).
$$
From the polar decomposition $r^*=u|r^*|=up^{-\frac12}$, we have: $r=p^{-\frac12}u^*$, 
and $r^{-1}=up^{\frac12}$. So we can write: 
$$
\mu(c)=\operatorname{Tr}(p^{\frac12}pp^{-\frac12}u^*c^Tu)
=\operatorname{Tr}(up^{\frac12}pp^{-\frac12}u^*c^T)=\operatorname{Tr}(r^{-1}prc^T).
$$
Since $p={r^*}^{-1}r^{-1}$, we thus have:
$$
\mu(c)=\operatorname{Tr}(r^{-1}{r^*}^{-1}c^T)=\operatorname{Tr}\bigl((r^{-1}{r^*}^{-1}c^T)^T\bigr)
=\operatorname{Tr}\bigl(c(r^{-1}{r^*}^{-1})^T\bigr)=\operatorname{Tr}(qc),
$$
where $q=(r^{-1}{r^*}^{-1})^T=\bigl((r^*r)^{-1}\bigr)^T$.

(2). We can check the result directly, by using the characterization given in (1) and by imitating 
the proof of Proposition~\ref{weightTrp}.  Or, we can compute the following:
$$
\sigma^{\mu}_t(c)=(R\circ\sigma^{\nu}_{-t}\circ R^{-1})(c)=(up^{-it}u^*c^Tup^{it}u^*)^T
=(up^{it}u^*)^Tc(up^{-it}u^*)^T.
$$
But, $up^{it}u^*=(upu^*)^{it}=(up^{\frac12}p^{\frac12}u^*)^{it}=(r^{-1}{r^*}^{-1})^{it}
=(q^T)^{it}$.  It follows that $(up^{it}u^*)^T=q^{it}$.  Also, $(up^{-it}u^*)^T=q^{-it}$. 
In this way, we can see that $\sigma^{\mu}_t(c)=q^{it}cq^{-it}$.
\end{proof}

The next proposition describes the (densely-defined) anti-homomorphism maps $\gamma_B$, 
$\gamma_B^{-1}$, $\gamma_C$, $\gamma_C^{-1}$ that we saw in Sections~2 and 3.

\begin{prop}
\begin{enumerate}
  \item For $b\in{\mathcal D}(\sigma^{\nu}_{i/2})$, write $\gamma_B(b):=(R\circ\sigma^{\nu}_{i/2})(b)
  =(r^*b{r^*}^{-1})^T$.  We have:
$$
E(b\otimes1)=E\bigl(1\otimes\gamma_B(b)\bigr)=E\bigl(1\otimes(r^*b{r^*}^{-1})^T\bigr).
$$
  \item For $c\in\operatorname{Ran}(\gamma_B)={\mathcal D}(\sigma^{\mu}_{i/2})$, 
write $\gamma_B^{-1}(c)=(R^{-1}\circ\sigma^{\mu}_{i/2})(c)={r^*}^{-1}c^Tr^*$.  We have:
$$
E(1\otimes c)=E\bigl(\gamma_B^{-1}(c)\otimes1\bigr)=E\bigl({r^*}^{-1}c^Tr^*\otimes1\bigr).
$$
  \item For $c\in{\mathcal D}(\sigma^{\mu}_{-i/2})$, write $\gamma_C(c):=(R^{-1}
\circ\sigma^{\mu}_{-i/2})(c)=(\sigma^{\nu}_{i/2}\circ R^{-1})(c)=rc^Tr^{-1}$.  We have: 
$$
(1\otimes c)E=\bigl(\gamma_C(c)\otimes1\bigr)E=\bigl(rc^Tr^{-1}\otimes1\bigr)E.
$$
  \item For $b\in\operatorname{Ran}(\gamma_C)={\mathcal D}(\sigma^{\nu}_{-i/2})$, 
write $\gamma_C^{-1}(b)=(R\circ\sigma^{\nu}_{-i/2})(b)=(r^{-1}br)^T$.  We have:
$$
(b\otimes1)E=\bigl(1\otimes\gamma_C^{-1}(b)\bigr)E=\bigl(1\otimes(r^{-1}br)^T\bigr)E.
$$
  \item The maps $\gamma_B$, $\gamma_B^{-1}$, $\gamma_C$, $\gamma_C^{-1}$ are 
anti-homomorphisms such that 
$$
{\text {$\gamma_C\bigl(\gamma_B(b)^*\bigr)^*=b$, for $b\in{\mathcal D}(\sigma^{\nu}_{i/2})$; 
and $\gamma_B\bigl(\gamma_C(c)^*\bigr)^*=c$, for $c\in{\mathcal D}(\sigma^{\mu}_{-i/2})$.}}
$$
\end{enumerate}
\end{prop}

\begin{proof}
See Propositions~\ref{gammagamma-1}, \ref{sepidgamma}, \ref{sepidgamma'}, \ref{gamma*}, 
\ref{gammagamma'eqns}.
\end{proof}

\section{Separability idempotents in von Neumann algebras}

\subsection{Separability idempotent in the von Neumann algebra setting}

It is possible to consider the notion of a separability idempotent in the setting of von Neumann 
algebras.  We will do this here.  Later in the section, we will also explore how the $C^*$-algebra 
approach and the von Neumann algebra approach are related.

We begin by considering a triple $(E,N,\tilde{\nu})$, where
\begin{itemize}
  \item $N$ is a von Neumann algebra;
  \item $\tilde{\nu}$ is an n.s.f.~weight on $N$, together with its associated modular automorphism 
group $(\sigma^{\tilde{\nu}}_t)$;
  \item $E$ is a self-adjoint idempotent (so a projection) element contained in the von Neumann 
algebra tensor product $N\otimes L$, where $L$ is a von Neumann algebra such that there exists 
a ${}^*$-anti-isomorphism $R:N\to L$. 
\end{itemize}

Consider the triple $(E,N,\tilde{\nu})$ as above.  As in Section~2, we consider $L$ and $R:N\to L$ 
to be implicitly understood.  We now give the definition of the separability triple:

\begin{defn}\label{vnseparabilitytriple}
We will say that $(E,N,\tilde{\nu})$ is a {\em separability triple\/}, if the following conditions hold:
\begin{enumerate}
  \item $(\tilde{\nu}\otimes\operatorname{id})(E)=1$
  \item For $b\in{\mathcal D}(\sigma^{\tilde{\nu}}_{i/2})$, we have: $(\tilde{\nu}\otimes\operatorname{id})
\bigl(E(b\otimes1)\bigr)=R\bigl(\sigma^{\tilde{\nu}}_{i/2}(b)\bigr)$.
\end{enumerate}
If $(E,N,\tilde{\nu})$ forms a separability triple, then we say $E$ is a {\em separability idempotent\/}.
\end{defn}

Let $(E,N,\tilde{\nu})$ be a separability triple.  Then, using the ${}^*$-anti-isomorphism $R:N\to L$, 
we can define the n.s.f.~weight $\tilde{\mu}$ on $L$, given by $\tilde{\mu}:=\tilde{\nu}\circ R^{-1}$. 
The arguments given in Sections 2 and 3 will all go through, including the existence of the 
densely-defined anti-homomorphisms $\gamma_N$, $\gamma_L$ and $\gamma_N^{-1}$, 
$\gamma_L^{-1}$:

\begin{prop}
We have the maps $\gamma_N=R\circ\sigma^{\tilde{\nu}}_{i/2}=\sigma^{\tilde{\mu}}_{-i/2}\circ R$ 
and $\gamma_L=\sigma^{\tilde{\nu}}_{i/2}\circ R^{-1}=R^{-1}\circ\sigma^{\tilde{\mu}}_{-i/2}$.  They 
are closed and densely-defined, injective anti-homomorphisms, having dense ranges.  We can 
also consider $\gamma_N^{-1}$ and $\gamma_L^{-1}$, having similar properties.  In addition, 
we have: 
\begin{enumerate}
  \item For $b\in{\mathcal D}(\sigma^{\tilde{\nu}}_{i/2})$, we have: $E(b\otimes1)
=E\bigl(1\otimes\gamma_N(b)\bigr)$.
  \item For $c\in{\mathcal D}(\sigma^{\tilde{\mu}}_{i/2})$, we have: $E(1\otimes c)
=E\bigl(\gamma_N^{-1}(c)\otimes1\bigr)$.
  \item For $c\in{\mathcal D}(\sigma^{\tilde{\mu}}_{-i/2})$, we have: $(1\otimes c)E
=\bigl(\gamma_L(c)\otimes1)E$.
  \item For $b\in{\mathcal D}(\sigma^{\tilde{\nu}}_{-i/2})$, we have: $(b\otimes1)E
=\bigl(1\otimes\gamma_L^{-1}(b)\bigr)E$.
\end{enumerate}
\end{prop}

We can also show that $E$ is full.  The proof is done in essentially the same way as in 
Proposition~\ref{Efull}.  The only difference is that while the sets were norm dense in the 
$C^*$-algebra setting, now the sets are strongly${}^*$-dense in the current setting:

\begin{prop}\label{vnEfull}
The separability idempotent $E\in N\otimes L$ is ``full'', in the sense that

$\bigl\{(\theta\otimes\operatorname{id})(E(b\otimes1)):b\in N,\theta\in N_*\bigr\}$ is dense in $L$,

$\bigl\{(\theta\otimes\operatorname{id})((b\otimes1)E):b\in N,\theta\in N_*\bigr\}$ is dense in $L$,

$\bigl\{(\operatorname{id}\otimes\omega)((1\otimes c)E):c\in L,\omega\in L_*\bigr\}$ is dense in $N$,

$\bigl\{(\operatorname{id}\otimes\omega)(E(1\otimes c)):c\in L,\omega\in L_*\bigr\}$ is dense in $N$.
\end{prop}

Other results in Sections 2 and 3 also have corresponding results, including the following: 

\begin{prop}\label{EresultsVN}
Let the notation be as above, and as before, denote by $\sigma$ the flip map on $N\otimes L$. 
Then
\begin{enumerate}
  \item $(\sigma^{\tilde{\nu}}_t\otimes\sigma^{\tilde{\mu}}_{-t})(E)=E$, for all $t\in\mathbb{R}$.
  \item $(\gamma_L\otimes\gamma_N)(\sigma E)=E$ and $(\gamma_N\otimes\gamma_L)(E)
=\sigma E$.
  \item $(R^{-1}\otimes R)(\sigma E)=E$ and $(R\otimes R^{-1})(E)=\sigma E$.
\end{enumerate}
\end{prop}

\begin{proof}
See Propositions \ref{sigmasigmaE} and \ref{sigmaE}.
\end{proof}

\subsection{Comparison with the $C^*$-algebraic notion}

The results above seem to indicate that the $C^*$-algebra approach and the von Neumann algebra 
approach to the separability idempotent are equivalent.  To prove this, let us first begin with the 
$C^*$-algebra framework.

Suppose $(E,B,\nu)$ is a separability triple in the $C^*$-algebra setting, with the $C^*$-algebra $C$ 
and the ${}^*$-ani-isomorphism $R:B\to C$ understood, and $\mu=\nu\circ R^{-1}$.  Let $\pi_{\nu}$ and 
$\pi_{\mu}$ be the GNS representations of $B$ and $C$, on ${\mathcal H}_{\nu}$ and ${\mathcal H}_{\mu}$, 
respectively.  For convenience, we may regard $B=\pi_{\nu}(B)\,\subseteq{\mathcal B}({\mathcal H}_{\nu})$, 
and similarly $C=\pi_{\mu}(C)\,\subseteq{\mathcal B}({\mathcal H}_{\mu})$.  Then we have 
$E=(\pi_{\nu}\otimes\pi_{\mu})(E)\,\subseteq{\mathcal B}({\mathcal H}_{\nu}\otimes{\mathcal H}_{\mu})$.

Consider $N:=\pi_{\nu}(B)''$.  Denote by $\tilde{\nu}$ the weight on $N$ lifted from $\nu$.  Since $\nu$ is 
a KMS weight, we know that $\tilde{\nu}$ should be an n.s.f. weight on $N$.  Similarly, define $L:=\pi_{\mu}(C)''$.  
We can perform the lifting of $\mu$ and obtain $\tilde{\mu}$, an n.s.f. weight on $L$.  We expect that our 
separability idempotent will be described at the level of these von Neumann algebras.  First, let us show that 
the ${}^*$-anti-isomorphism $R:B\to C$ extends to $\tilde{R}:N\to L$, also a ${}^*$-anti-isomorphism.

\begin{lem}\label{lemma51}
We have:
$$
C=\overline{\bigl\{(\nu\otimes\operatorname{id})((b_1^*\otimes1)E(b_2\otimes1)):b_1,b_2\in{\mathfrak N}_{\nu}\bigr\}}^{\|\ \|},
$$
$$
B=\overline{\bigl\{(\operatorname{id}\otimes\mu)((1\otimes c_1^*)E(1\otimes c_2)):c_1,c_2\in{\mathfrak N}_{\mu}\bigr\}}^{\|\ \|}.
$$
\end{lem}

\begin{proof}
First statement is a consequence of $E$ being ``full'' (Proposition~\ref{Efull}), together with the fact that ${\mathfrak N}_{\nu}$ 
is dense in $B$.  Note also that the functionals of the form $\nu(b_1^*\,\cdot\,b_2)$, $b_1,b_2\in{\mathfrak N}_{\nu}$, are dense 
in $B^*$.  Second statement is also a consequence of $E$ being full.
\end{proof}

\begin{lem}\label{lemma52}
\begin{enumerate}
 \item $b\in{\mathfrak N}_{\nu}$ if and only if $R(b)^*\in{\mathfrak N}_{\mu}$
 \item $c\in{\mathfrak N}_{\mu}$ if and only if $R^{-1}(c)^*\in{\mathfrak N}_{\nu}$
\end{enumerate}
\end{lem}

\begin{proof}
Use $\nu=\mu\circ R$.  We have: $\nu(b^*b)=\mu\bigl(R(b^*b)\bigr)=\mu\bigl(R(b)R(b)^*\bigr)$, because $R$ is 
a ${}^*$-anti-isomorphism.  So $b\in{\mathfrak N}_{\nu}$ if and only if $R(b)^*\in{\mathfrak N}_{\mu}$.  Similar 
proof for the other statement.
\end{proof}

\begin{prop}\label{Rextends}
The ($C^*$-algebraic) ${}^*$-anti-isomorphism $R:B\to C$ extends to $\tilde{R}:N\to L$, which is also a (von Neumann algebraic) 
${}^*$-anti-isomorphism.
\end{prop}

\begin{proof}
From Lemma~\ref{lemma51}, we know the elements $(\operatorname{id}\otimes\mu)((1\otimes c_1^*)E(1\otimes c_2))$, 
$c_1,c_2\in{\mathfrak N}_{\mu}$, generate $B$.  Under $R$, since $\mu=\nu\circ R^{-1}$, we have:
\begin{align}
&R\bigl((\operatorname{id}\otimes\mu)((1\otimes c_1^*)E(1\otimes c_2))\bigr)
=(\operatorname{id}\otimes\nu)\bigl((R\otimes R^{-1})((1\otimes c_1^*)E(1\otimes c_2))\bigr) \notag \\
&=(\operatorname{id}\otimes\nu)\bigl((1\otimes R^{-1}(c_2))[(R\otimes R^{-1})(E)](1\otimes R^{-1}(c_1)^*)\bigr) \notag \\
&=(\nu\otimes\operatorname{id})\bigl(R^{-1}(c_2)\otimes1)E(R^{-1}(c_1)^*\otimes1)\bigr), 
\notag
\end{align}
using the fact that $(R\otimes R^{-1})(E)=\sigma E$ (Corollary of Proposition~3.9).  Similarly, we also have: 
$$
R^{-1}:(\nu\otimes\operatorname{id})\bigl((b_1^*\otimes1)E(b_2\otimes1)\bigr)\mapsto
(\operatorname{id}\otimes\mu)\bigl((1\otimes R(b_2))E(1\otimes R(b_1)^*)\bigr).
$$

Recall that $N=\pi_{\nu}(B)''=\bigl\{(\operatorname{id}\otimes\mu)((1\otimes c_1^*)E(1\otimes c_2)):c_1,c_2\in{\mathfrak N}_{\mu}\bigr\}''$
and $L=\pi_{\mu}(C)''=\bigl\{(\nu\otimes\operatorname{id})((b_1^*\otimes1)E(b_2\otimes1)):b_1,b_2\in{\mathfrak N}_{\nu}\bigr\}''$. 
In view of Lemma~\ref{lemma52}, looking at the generators of $N$ and $L$, it is now clear that the $R$ map extends to an 
injective map $\tilde{R}$ from $N$ onto $L$.  The preservation of the ${}^*$-structure and the anti-multiplicativity of $\tilde{R}$ are 
all straightforward because of the corresponding properties of $R$.
\end{proof}

\begin{rem}
For convenience, we may still write $R$ instead of $\tilde{R}$, with its inverse $R^{-1}:L\to N$ also a ${}^*$-anti-isomorphism. 
It is evident that we will have $\tilde{\mu}=\tilde{\nu}\circ R^{-1}$.  
\end{rem}

Finally, observe that $E\in M(B\otimes C)\subseteq N\otimes L$.  Since $B$, $C$ are dense 
in $N$, $L$, and since $\tilde{\nu}$, $\tilde{\mu}$ naturally extend $\nu$, $\mu$, it follows that 
$(E,N,\tilde{\nu})$ forms a separability triple in the sense of Definition~\ref{vnseparabilitytriple}.

Let us next consider the problem of going back.  So we begin with a separability triple 
$(E,N,\tilde{\nu})$ in the von Neumann algebra setting, again with the von Neumann algebra $L$ 
and the anti-isomorphism $R:N\to L$ understood.  Also recall that $\tilde{\mu}=\tilde{\nu}\circ R^{-1}$. 

Let ${\mathcal T}_{\tilde{\mu}}$ be the Tomita ${}^*$-algebra, which is dense in ${\mathfrak N}_{\tilde{\mu}}$. 
Let $c_1,c_2\in{\mathcal T}_{\tilde{\mu}}$.  Observe: 
$$
(\operatorname{id}\otimes\tilde{\mu})\bigl((1\otimes c_2^*)E(1\otimes c_1)\bigr)
=\bigl(\operatorname{id}\otimes\langle\cdot\,\Lambda_{\tilde{\mu}}(c_1),\Lambda_{\tilde{\mu}}(c_2)
\rangle\bigr)(E)=(\operatorname{id}\otimes\omega_{\xi,\eta})(E),
$$
where $\xi=\Lambda_{\tilde{\mu}}(c_1)$, $\eta=\Lambda_{\tilde{\mu}}(c_2)$.  Considering the result 
of Proposition~\ref{vnEfull}, we see that the $(\operatorname{id}\otimes\omega_{\xi,\eta})(E)$ are 
strongly${}^*$-dense in the von Neumann algebra $N$.  This suggests us to consider the norm closure 
of the space spanned by the elements $(\operatorname{id}\otimes\omega_{\xi,\eta})(E)$, for 
$\xi,\eta\in{\mathcal H}_{\tilde{\mu}}$.  Namely,
\begin{equation}\label{(BfromN)}
B:=\overline{\bigl\{(\operatorname{id}\otimes\omega_{\xi,\eta})(E):\xi,\eta\in{\mathcal H}_{\tilde{\mu}}
\bigr\}}^{\|\ \|}=\overline{\bigl\{(\operatorname{id}\otimes\omega)(E):\omega\in{\mathcal B}
({\mathcal H}_{\tilde{\mu}})_*\bigr\}}^{\|\ \|}.
\end{equation}
Similarly, define:
\begin{equation}\label{(CfromL)}
C:=\overline{\bigl\{(\theta\otimes\operatorname{id})(E):
\theta\in{\mathcal B}({\mathcal H}_{\tilde{\nu}})_*\bigr\}}^{\|\ \|}.
\end{equation}

We gather some results on these subspaces $B\,(\subseteq N)$ and $C\,(\subseteq L)$.

\begin{prop}\label{BCfromNL}
Let $(E,N,\tilde{\nu})$ be a separability triple in the von Neumann algebraic sense of 
Definition~\ref{vnseparabilitytriple}, so that $E\in N\otimes L$ is a separability idempotent. 
Consider the subspace $B$ of $N$, given in Equation~\eqref{(BfromN)}, and the subspace 
$C$ of $L$, given in Equation~\eqref{(CfromL)}.  Then we have:
\begin{enumerate}
  \item ${\mathcal D}(\gamma_N)\cap B$ is dense in $B$, and $\gamma_N$ restricted to this 
space has a dense range in $L$.  
  \item $B$ is a ${}^*$-subalgebra of $N$.
  \item ${\mathcal D}(\gamma_L)\cap C$ is dense in $C$, and $\gamma_L$ restricted to this 
space has a dense range in $N$.  
  \item $C$ is a ${}^*$-subalgebra of $L$.
\end{enumerate}
\end{prop}

\begin{proof}
(1). Let  $c_1,c_2\in{\mathcal T}_{\tilde{\mu}}$, and consider $(\operatorname{id}\otimes
\omega_{\Lambda_{\tilde{\mu}}(c_1),\Lambda_{\tilde{\mu}}(c_2)})(E)\in B$.    But, 
\begin{align}\label{(BCfromNLeqn1)}
(\operatorname{id}\otimes\omega_{\Lambda_{\tilde{\mu}}(c_1),\Lambda_{\tilde{\mu}}(c_2)})(E)
&=(\operatorname{id}\otimes\tilde{\mu})\bigl((1\otimes c_2^*)E(1\otimes c_1)\bigr) 
\notag \\
&=(\operatorname{id}\otimes\tilde{\mu})\bigl(E[1\otimes c_1\sigma^{\tilde{\mu}}_{-i}(c_2^*)]\bigr)
=\gamma_N^{-1}\bigl(c_1\sigma^{\tilde{\mu}}_{-i}(c_2^*)\bigr).
\end{align}
So we see that $(\operatorname{id}\otimes\omega_{\Lambda_{\tilde{\mu}}(c_1),
\Lambda_{\tilde{\mu}}(c_2)})(E)\in{\mathcal D}(\gamma_N)\cap B$.  We know from 
Equation~\eqref{(BfromN)} that such elements span a dense subspace of $B$.  In addition, 
the elements of the form $c_1\sigma^{\tilde{\mu}}_{-i}(c_2^*)$, $c_1,c_2\in{\mathcal T}_{\tilde{\mu}}$, 
are known to be dense in $L$.  As a consequence, we see that under $\gamma_N$, the space 
${\mathcal D}(\gamma_N)\cap B$ is sent to a dense subspace in $L$.

(2). Let $b\in{\mathcal D}(\gamma_N)\cap B$, and consider $\tilde{b}=(\operatorname{id}
\otimes\omega)(E)$, for $\omega\in{\mathcal B}({\mathcal H}_{\tilde{\mu}})_*$.  Such elements 
are dense in $B$.  Observe that
$$
\tilde{b}b=(\operatorname{id}\otimes\omega)(E)\,b
=(\operatorname{id}\otimes\omega)\bigl(E(b\otimes1)\bigr)
=(\operatorname{id}\otimes\omega)\bigl(E(1\otimes\gamma_N(b))\bigr)
=(\operatorname{id}\otimes\theta)(E),
$$
where $\theta=\omega\bigl(\cdot\,\gamma_N(b)\bigr)$.  This shows that $B$ is closed under 
multiplication.  Meanwhile, observe also that $\bigl[(\operatorname{id}\otimes\omega)(E)\bigr]^*
=(\operatorname{id}\otimes\bar{\omega})(E)$, by the self-adjointness of $E$.  So we see that
$B$ is a ${}^*$-subalgebra of $N$.

(3), (4). Proof is similar to that of (1), (2).
\end{proof}

\begin{cor}
\begin{enumerate}
  \item $B$ is a $C^*$-subalgebra of $N$, and $N$ is the von Neumann algebra closure of $B$.
  \item $C$ is a $C^*$-subalgebra of $L$, and $L$ is the von Neumann algebra closure of $C$.
\end{enumerate}
\end{cor}

\begin{proof}
For $B$, being a norm-closed ${}^*$-subalgebra of $N$, that is strongly ${}^*$-dense in $N$, 
the result follows immediately.  Similar for $C$ in $L$.
\end{proof}

In the next proposition, we show that $E\in M(B\otimes C)$ and that $E$ is a separability 
idempotent:

\begin{prop}
Let $E\in N\otimes L$ be a separability idempotent in the von Neumann algebraic sense, 
and consider the $C^*$-subalgebras $B$ of $N$ and $C$ of $L$, as defined above.  Denote 
by $\nu$ the weight on $B$, restricted from $\tilde{\nu}$ on $N$, and similarly, consider 
$\mu$ on $C$, restricted from $\tilde{\mu}$ on $L$.  Then:
\begin{enumerate}
  \item The $\sigma^{\tilde{\nu}}_t$, $t\in\mathbb{R}$, leaves $B$ invariant.  So we can consider 
$\sigma^{\nu}_t:=\sigma^{\tilde{\nu}}_t|_B$, for $t\in\mathbb{R}$.  Similar for $(\sigma^{\mu}_t)$, 
on $C$.
  \item $\nu$ is a KMS weight on $B$, equipped with the (norm-continuous) automorphism 
group $(\sigma^{\nu}_t)$.  Similarly, $\mu$ is a KMS weight on $C$, equipped with the 
automorphism group $(\sigma^{\mu}_t)$. 
  \item $E\in M(B\otimes C)$.
  \item $(E,B,\nu)$ forms a separability triple in the $C^*$-algebraic sense of 
Definition~\ref{separabilitytriple}.
\end{enumerate}
\end{prop}

\begin{proof}
(1). Consider $(\operatorname{id}\otimes\omega)(E)\in B$, where $\omega\in{\mathcal B}
({\mathcal H}_{\tilde{\mu}})_*$.  For any $t\in\mathbb{R}$, we know from Proposition~\ref{EresultsVN} 
that $(\sigma^{\tilde{\nu}}_{-t}\otimes\sigma^{\tilde{\mu}}_{t})(E)=E$.  So we have:
\begin{equation}\label{(sigmanu)}
\sigma^{\tilde{\nu}}_t\bigl((\operatorname{id}\otimes\omega)(E)\bigr)
=\sigma^{\tilde{\nu}}_t\bigl((\operatorname{id}\otimes\omega)[(\sigma^{\tilde{\nu}}_{-t}
\otimes\sigma^{\tilde{\mu}}_{t})(E)]\bigr)
=\bigl(\operatorname{id}\otimes(\omega\circ\sigma^{\tilde{\mu}}_t)\bigr)(E).
\end{equation}
Since $B=\overline{\bigl\{(\operatorname{id}\otimes\omega)(E):\omega\in{\mathcal B}
({\mathcal H}_{\tilde{\mu}})_*\bigr\}}^{\|\ \|}$, this means that $\sigma^{\tilde{\nu}}_t(B)=B$, 
for all $t\in\mathbb{R}$.  By a similar argument, we can show that $\sigma^{\tilde{\mu}}_t(C)=C$, 
for all $t\in\mathbb{R}$.

(2). Let $\nu=\tilde{\nu}|_B$.  It is a faithful weight because $\tilde{\nu}$ is.  To see that $\nu$ 
is semi-finite, consider  $(\operatorname{id}\otimes\omega_{\Lambda_{\tilde{\mu}}(c_1),
\Lambda_{\tilde{\mu}}(c_2)})(E)\in B$, for $c_1,c_2\in{\mathcal T}_{\tilde{\mu}}$.  We know that 
such elements span a dense subspace of $B$.  By Equation~\eqref{(BCfromNLeqn1)}, we have: 
\begin{align}
\nu\bigl((\operatorname{id}\otimes\omega_{\Lambda_{\tilde{\mu}}(c_1),
\Lambda_{\tilde{\mu}}(c_2)})(E)\bigr)
&=\tilde{\nu}\bigl((\operatorname{id}\otimes\omega_{\Lambda_{\tilde{\mu}}(c_1),
\Lambda_{\tilde{\mu}}(c_2)})(E)\bigr)
=\tilde{\nu}\bigl(\gamma_N^{-1}(c_1\sigma^{\tilde{\mu}}_{-i}(c_2^*))\bigr)
\notag \\
&=\tilde{\mu}\bigl(c_1\sigma^{\tilde{\mu}}_{-i}(c_2^*)\bigr)\,<\infty.
\notag
\end{align}
In the last line, we used the fact that $\gamma_N^{-1}=R^{-1}\circ\sigma^{\tilde{\mu}}_{-i/2}$ 
(analogous to Proposition~\ref{gammagamma'eqns}), so $\tilde{\nu}\circ\gamma_N^{-1}
=\tilde{\mu}\circ\sigma^{\tilde{\mu}}_{-i/2}=\tilde{\mu}$.  In this way, we have shown 
that $\nu$ is valid on the dense subspace $\bigl\{(\operatorname{id}\otimes
\omega_{\Lambda_{\tilde{\mu}}(c_1),\Lambda_{\tilde{\mu}}(c_2)})(E): c_1,c_2
\in{\mathcal T}_{\tilde{\mu}}\bigr\}$, which forms a core for $\nu$.

Meanwhile, since $\sigma^{\tilde{\nu}}_t(B)=B$, we may write $\sigma^{\nu}_t
=\sigma^{\tilde{\nu}}_t|_B$, for $t\in\mathbb{R}$.  With $t\mapsto\sigma^{\tilde{\nu}}_t$ 
being strongly continuous, and by Equation~\eqref{(sigmanu)}, we have that for $b\in B$, 
the function $t\mapsto\sigma^{\nu}_t(b)$ is norm-continuous.  From the properties of the 
modular automorphism group $(\sigma^{\tilde{\nu}}_t)_{t\in\mathbb{R}}$ for $\tilde{\nu}$, 
we can show easily the analogous properties for $(\sigma^{\nu}_t)_{t\in\mathbb{R}}$, and 
this means that $\nu$ is a KMS weight on the $C^*$-algebra $B$.  By construction, it is 
clear that the $W^*$-lift of $\nu$ will be just $\tilde{\nu}$ on $N$.

Similarly, if we let $\mu=\tilde{\mu}|_C$, then it is a KMS weight on the $C^*$-algebra 
$C$, together with the one-parameter group of automorphisms $(\sigma^{\mu}_t)_{t\in\mathbb{R}}$, 
where $\sigma^{\mu}_t=\sigma^{\tilde{\mu}}_t|_C$.  The $W^*$-lift of $\mu$ is $\tilde{\mu}$ on $L$.  

(3). We already know that $E$ is a self-adjoint idempotent contained in $N\otimes L$. 
Therefore, by Proposition~\ref{vnEfull}, we have for any $\varepsilon>0$ and for any 
$\xi\in{\mathcal H}_{\tilde{\nu}}$, $\eta\in{\mathcal H}_{\tilde{\mu}}$, we can find a finite 
number of elements $p_1,p_2,\dots,p_n\in B$ of the form $(\operatorname{id}\otimes\omega)(E)$, 
$\omega\in{\mathcal B}({\mathcal H}_{\tilde{\mu}})_*$ and $q_1,q_2,\dots,q_n\in C$ of the 
form $(\theta\otimes\operatorname{id})(E)$, $\theta\in{\mathcal B}({\mathcal H}_{\tilde{\nu}})_*$, 
such that 
$$
\bigl\|(E-\sum_{k=1}^np_k\otimes q_k)(\xi\otimes\eta)\bigr\|<\varepsilon.
$$
Since $N$ (also $B$) acts on ${\mathcal H}_{\tilde{\nu}}$ in a non-degenerate way, and similarly 
for $L$ (also $C$) on ${\mathcal H}_{\tilde{\mu}}$, without loss of generality we may let $\xi=
\Lambda_{\nu}(b_0)=\Lambda_{\tilde{\nu}}(b_0)$ and $\eta=\Lambda_{\mu}(c_0)=\Lambda_{\tilde{\mu}}(c_0)$, 
where $b_0\in{\mathfrak N}_{\nu}$, $c_0\in{\mathfrak N}_{\mu}$.  
Then:
$$
\bigl\|(\Lambda_{\tilde{\nu}}\otimes\Lambda_{\tilde{\mu}})\bigl(E(b_0\otimes c_0)
-\sum_{k=1}^n(p_k\otimes q_k)(b_0\otimes c_0)\bigr)\bigr\|<\varepsilon.
$$
Since $b_0$, $c_0$ are arbitrary, while ${\mathfrak N}_{\nu}$ is dense in $B$ and 
${\mathfrak N}_{\mu}$ is dense in $C$, we can see that $E$ is contained in the closure of 
$B\otimes C$ under the strict topology.  In other words, $E\in M(B\otimes C)$.

(4). From (3), we saw that $E\in M(B\otimes C)$.  From (1),\,(2), we saw that $\nu$ is a KMS weight 
on $B$, together with its associated norm-continuous automorphism group $(\sigma^{\nu}_t)$.

Meanwhile, note that for $(\operatorname{id}\otimes\omega)(E)\in B\,(\subseteq N)$, 
by Proposition~\ref{EresultsVN}\,(3), we have:
\begin{align}
R\bigl((\operatorname{id}\otimes\omega)(E)\bigr)
&=R\bigl((\omega\otimes\operatorname{id})(\sigma E)\bigr)  \notag \\
&=R\bigl((\omega\otimes\operatorname{id})[(R\otimes R^{-1})(E)]\bigr)
=\bigl((\omega\circ R)\otimes\operatorname{id}\bigr)(E)\in C.
\notag
\end{align}
This shows that the ${}^*$-anti-isomorphism $R:N\to L$ restricts to $R:B\to C$.  This 
will be also a ${}^*$-anti-isomorphism.  Since $R:B\to C$ is a restriction of $R:N\to L$ and since 
$\sigma^{\nu}_t=\sigma^{\tilde{\nu}}_t|_B$, the conditions of Definition~\ref{vnseparabilitytriple} 
for $(E,N,\tilde{\nu})$ immediately give rise to the conditions of Definition~\ref{separabilitytriple}. 
This means that $(E,B,\nu)$ is a separability triple in the $C^*$-algebraic sense.  And, the 
associated $\gamma_B$, $\gamma_C$ maps are none other than the restrictions of $\gamma_N$ 
and $\gamma_L$ to $B$ and $C$, respectively, which have already appeared in (1),\,(3) of 
Proposition~\ref{BCfromNL}.
\end{proof}

\section{The nature of the base $C^*$-algebra $B$}

The requirements for $(E,B,\nu)$ being a separability triple means that some possible restrictions 
will be needed on the pair $(B,\nu)$.  This question was considered in the algebraic setting (see \cite{VDsepid}), 
and the conclusion in that framework was that $B$ has to be a direct sum of matrix algebras.  While 
the current situation is different, we do expect that some similar conditions on the base $C^*$-algebra $B$ 
will exist.  We wish to explore this question in this section.

To begin, assume that $B$ is a $C^*$-algebra, $\nu$ a KMS weight on $B$, such that the pair 
$(B,\nu)$ satisfies the relevant conditions that give us a {\em separability idempotent\/} 
$E\in M(B\otimes C)$.  Here $C$ is the opposite $C^*$-algebra of $B$.

\subsection{The subalgebra $BbB$ for a fixed $b\in B$}

Let us fix $b\in B$.  By Proposition~\ref{BtensorC}, we know that $(b\otimes1)E$ is contained 
in the $C^*$-tensor product $B\otimes C$. So, for any $k\in\mathbb{N}$, we can find a finite number 
of elements $b^{(k)}_i\in B$, $c^{(k)}_i\in C$ ($i=1,2,\dots,N_k$), such that
$$
\left\|(b\otimes1)E-\sum_{i=1}^{N_k}b^{(k)}_i\otimes c^{(k)}_i\right\|\,<\,\frac1k.
$$

Next, let $\tilde{b}\in B$ be arbitrary and let $\varepsilon>0$.  By Proposition~\ref{Efull}
(since $E$ is ``full''), we can find $\tilde{c}\in C$, $\tilde{\omega}\in C^*$, such that 
$\bigl\|\tilde{b}-(\operatorname{id}\otimes\tilde{\omega})(E(1\otimes\tilde{c}))\bigr\|
<\frac{\varepsilon}{2\|b\|}$.  Then we have:
$$
\bigl\|b\tilde{b}-b(\operatorname{id}\otimes\tilde{\omega})(E(1\otimes\tilde{c}))\bigr\|
<\frac{\varepsilon}{2}.
$$

Consider $k\in\mathbb{N}$ such that $\frac1k<\dfrac{\varepsilon}{2\|\tilde{\omega}\|\|\tilde{c}\|}$, 
and as above, find corresponding $b^{(k)}_i\in B$, $c^{(k)}_i\in C$ ($i=1,2,\dots,N_k$). Then:
\begin{align}
&\bigl\|b(\operatorname{id}\otimes\tilde{\omega})(E(1\otimes\tilde{c}))-\sum_{i=1}^{N_k}
b^{(k)}_i\tilde{\omega}(c^{(k)}_i\tilde{c})\bigr\|   \notag \\
&=\bigl\|(\operatorname{id}\otimes\tilde{\omega})\bigl([(b\otimes1)E-
\sum_{i=1}^{N_k}b^{(k)}_i\otimes c^{(k)}_i](1\otimes\tilde{c})\bigr)\bigr\|
<\frac1k\cdot\|\tilde{\omega}\|\|\tilde{c}\|<\frac{\varepsilon}{2}.
\notag
\end{align}
Combine this with the earlier inequality to obtain (by triangle inequality):
$$
\bigl\|b\tilde{b}-\sum_{i=1}^{N_k}b^{(k)}_i\tilde{\omega}(c^{(k)}_i\tilde{c})\bigr\|<\frac{\varepsilon}{2}
+\frac{\varepsilon}{2}=\varepsilon.
$$

In the above, each $\tilde{\omega}(c^{(k)}_i\tilde{c})\in\mathbb{C}$.  So we have shown that 
any $b\tilde{b}$ is within $\varepsilon$-distance from $\operatorname{span}(b^{(k)}_i:
k\in\mathbb{N},i=1,2,\dots,N_k)$.  And, this result is true for arbitrary $\tilde{b}\in B$ 
and any $\varepsilon>0$.  We thus have: 
$$
bB\subseteq\overline{\operatorname{span}(b^{(k)}_i:k\in\mathbb{N},i=1,2,\dots,N_k)}^{\|\ \|}.
$$

From the countable generating set $\{b^{(k)}_i\}$ for $bB$, consider a maximal linearly independent 
(basis) subset $\{b_{\lambda}\}_{\lambda\in\Lambda}$, indexed by a countable set $\Lambda$. 
We obtain the following result:

\begin{prop}\label{bBgenerators}
\begin{enumerate}
  \item For a fixed $b\in B$, we can find a countable collection of elements $\{b_{\lambda}\}_{\lambda
\in\Lambda}$ contained in $B$, such that
$$
bB\subseteq\overline{\operatorname{span}(b_{\lambda}:\lambda\in\Lambda)}^{\|\ \|}.
$$
  \item For any $k\in\mathbb{N}$, we can find a finite set $F_k\subset\Lambda$ and a finite number 
of elements $\{c_{\lambda}\}_{\lambda\in F_k}$ in $C$, such that
$$
\bigl\|(b\otimes1)E-\sum_{\lambda\in F_k}b_{\lambda}\otimes c_{\lambda}\bigr\|<\frac1k.
$$
  \item For any $z\in bB$ and any $\varepsilon>0$, we can find a finite set $\tilde{F}\subset\Lambda$ and  
a finite number of elements $\{\tilde{c}_{\lambda}\}_{\lambda\in\tilde{F}}$ in $C$, such that
$$
\bigl\|(z\otimes1)E-\sum_{\lambda\in\tilde{F}}b_{\lambda}\otimes\tilde{c}_{\lambda}\bigr\|<\varepsilon.
$$
\end{enumerate}
\end{prop}

\begin{proof}
(1). Each generator $b^{(k)}_i$ for $bB$ may be expressed as a linear combination of the 
$\{b_{\lambda}\}$. So the result is clear.

(2). For any finite sum of the form $\sum_{k,i}b^{(k)}_i\otimes c^{(k)}_i$, we can find a finite set 
$F\subset\Lambda$ so that $\sum_{k,i}b^{(k)}_i\otimes c^{(k)}_i=\sum_{\lambda\in F}b_{\lambda}
\otimes c_{\lambda}$, where the $c_{\lambda}$ are appropriate elements in $C$.  In particular, 
for a fixed $k\in\mathbb{N}$, we can find a finite set $F_k\subset\Lambda$ and suitable elements 
$\{c_{\lambda}\}_{\lambda\in F_k}$, such that $\sum_{i=1}^{N_k}b^{(k)}_i\otimes c^{(k)}_i
=\sum_{\lambda\in F_k}b_{\lambda}\otimes c_{\lambda}$.  Then we have: 
$\bigl\|(b\otimes1)E-\sum_{\lambda\in F_k}b_{\lambda}\otimes c_{\lambda}\bigr\|<\frac1k$.

(3). Find $\tilde{b}\in B$ such that $\|z-b\tilde{b}\|<\frac{\varepsilon}2$.  Without loss of generality, 
we may assume that $\tilde{b}\in{\mathcal D}(\gamma_C^{-1})$, such that $(\tilde{b}\otimes1)E=
\bigl(1\otimes\gamma_C^{-1}(\tilde{b})\bigr)E$.  Note that $(b\tilde{b}\otimes1)E
=\bigl(b\otimes\gamma_C^{-1}(\tilde{b})\bigr)E
=\bigl(1\otimes\gamma_C^{-1}(\tilde{b})\bigr)(b\otimes1)E$.

Find $k\in\mathbb{N}$ such that $\frac1k<\dfrac{\varepsilon}{2\|\gamma_C^{-1}(\tilde{b})\|}$, 
then by (2) we have a finite subset $\tilde{F}=F_k\subset\Lambda$ such that:
\begin{align}
\bigl\|(b\tilde{b}\otimes1)E-\sum_{\lambda\in\tilde{F}}b_{\lambda}\otimes
\gamma_C^{-1}(\tilde{b})c_{\lambda}\bigr\|
&=\bigl\|\bigl(1\otimes\gamma_C^{-1}(\tilde{b})\bigr)[(b\otimes1)E-\sum_{\lambda\in\tilde{F}}
b_{\lambda}\otimes c_{\lambda}]\bigr\|   \notag \\
&<\frac1k\cdot\|\gamma_C^{-1}(\tilde{b})\|<\frac{\varepsilon}2.
\notag
\end{align}

Write $\tilde{c}_{\lambda}=\gamma_C^{-1}(\tilde{b})c_{\lambda}$.  Then, 
$\bigl\|(b\tilde{b}\otimes1)E
-\sum_{\lambda\in\tilde{F}}b_{\lambda}\otimes\tilde{c}_{\lambda}\bigr\|<\frac{\varepsilon}2$. 
Meanwhile, $\bigl\|(z\otimes1)E-(b\tilde{b}\otimes1)E\bigr|\le\|z-b\tilde{b}\|\|E\|
<\frac{\varepsilon}2$.  By the triangle inequality, we have:
$\bigl\|(z\otimes1)E-\sum_{\lambda\in\tilde{F}}b_{\lambda}\otimes\tilde{c}_{\lambda}\bigr\|
<\varepsilon$.
\end{proof}

Next, for a typical generator $b_{\lambda}$ (for a fixed $\lambda\in\Lambda$), we wish to explore 
the space $Bb_{\lambda}$.  Since $E(b_{\lambda}\otimes1)$ is contained in the $C^*$-tensor 
product $B\otimes C$, for any $m\in\mathbb{N}$ we can find a finite number of elements 
$b_{\lambda,j}^{(m)}\in B$, $c_{\lambda,j}^{(m)}\in C$ ($j=1,2,\dots,N_m$), such that
$$
\bigl\|E(b_{\lambda}\otimes1)-\sum_{j=1}^{N_m}b_{\lambda,j}^{(m)}\otimes c_{\lambda,j}^{(m)}\bigr\|
<\frac1m.
$$

From the set $\{b_{\lambda,j}^{(m)}:m\in\mathbb{N},j=1,2,\dots,N_m\}$, find a maximal linearly 
independent subset $\{b_{\lambda,q}\}_{q\in Q_{\lambda}}$, indexed by a countable set 
$Q_{\lambda}$.  Then we will have similar results as before:

\begin{prop}\label{Bb_lambdagenerators}
\begin{enumerate}
  \item $Bb_{\lambda}\subseteq\overline{\operatorname{span}(b_{\lambda,q}:
q\in Q_{\lambda})}^{\|\ \|}$.
  \item For any $y\in Bb_{\lambda}$ and any $\varepsilon>0$, we can find a finite set $\tilde{F}
\subset Q_{\lambda}$ and a finite number of elements $\{\tilde{c}_{\lambda}\}_{\lambda\in\tilde{F}}$ 
in $C$, such that
$$
\bigl\|(E(y\otimes1)-\sum_{q\in\tilde{F}}b_{\lambda,q}\otimes\tilde{c}_{\lambda}\bigr\|<\varepsilon.
$$
\end{enumerate}
\end{prop}

\begin{proof}
Proof is similar to that of Proposition~\ref{bBgenerators}.
\end{proof}

Proposition~\ref{Bb_lambdagenerators} is valid for any $b_{\lambda}$, for $\lambda
\in\Lambda$.  Meanwhile, it is evident that $BbB\subseteq\overline{\cup_{\lambda\in\Lambda}
Bb_{\lambda}}$.  So $BbB\subseteq\overline{\operatorname{span}
(b_{\lambda,q}:\lambda\in\Lambda,q\in Q_{\lambda})}^{\|\ \|}$. 
Again, similarly as above, find $\{b_p\}_{p\in P}$, indexed by a countable set $P$ and is 
a maximal linearly independent subset of the $\{b_{\lambda,q}\}$.  Then we have the following 
result:

\begin{theorem}\label{BbBgenerators}
\begin{enumerate}
  \item For a fixed $b\in B$, we can find a countable set of elements $\{b_p\}_{p\in P}$ in $B$, 
such that
$$
BbB\subseteq\overline{\operatorname{span}(b_p:p\in P)}^{\|\ \|}.
$$
  \item For any $x\in BbB$ and any $\varepsilon>0$, we can find a finite set $\tilde{F}\subset P$ 
and a finite number of elements $\{\tilde{c}_p\}_{p\in\tilde{F}}$ in $C$, such that
$$
\bigl\|(x\otimes1)E-\sum_{p\in\tilde{F}}b_p\otimes\tilde{c}_p\bigr\|<\varepsilon.
$$
\end{enumerate}
\end{theorem}

\begin{proof}
(1). This is a consequence of the previous paragraph.

(2). Let $x\in BbB$ and $\varepsilon>0$ be arbitrary. We can find a finite number of elements 
$b'_1,b''_1;b'_2,b''_2;\dots,b'_n,b''_n\in B$ such that 
$$
\bigl\|x-(b'_1bb''_1+b'_2bb''_2+\dots+b'_nbb''_n)\bigr\|<\frac{\varepsilon}3.
$$
Since $\gamma_C^{-1}$ is densely-defined in $B$, we may as well take each 
$b''_i\in{\mathcal D}(\gamma_C^{-1})$.  Then $(b'_ibb''_i\otimes1)E
=\bigl(b'_ib\otimes\gamma_C^{-1}(b''_i)\bigr)E
=\bigl(b'_i\otimes\gamma_C^{-1}(b''_i)\bigr)(b\otimes1)E$.

We know from Proposition~\ref{bBgenerators}\,(2) that the expression $(b\otimes1)E$ can be 
approximated arbitrarily closely by a finite sum of the form $\sum_{\lambda\in F_k}b_{\lambda}
\otimes c_{\lambda}$.  This means that for each $i$, the expression 
$(b'_ibb''_i\otimes1)E$ can be approximated arbitrarily closely by a finite sum of the form 
$\sum_{\lambda\in F_k}b'_ib_{\lambda}\otimes\gamma_C^{-1}(b''_i)c_{\lambda}$.  Meanwhile, 
by Proposition~\ref{Bb_lambdagenerators} and part (1) above that each $b'_ib_{\lambda}$ 
can be expressed as a finite linear combination in terms of the $\{b_p\}_{p\in P}$.  So we can 
find a finite set $P_i\subset P$ such that $\bigl\|(b'_ibb''_i\otimes1)E-\sum_{p\in P_i}
b_p\otimes c^{(i)}_p\bigr\|<\frac{\varepsilon}{3n}$, for some appropriate $\{c^{(i)}_p\}_{p\in P_i}
\in C$.  This can be done for each $(b'_ibb''_i\otimes1)E$, $i=1,2,\dots,n$. 

We are able to find a finite set $\tilde{F}\subset P_1\cup P_2\cup\dots\cup P_n\subset P$ 
and suitable elements $\{\tilde{c}_p\}_{p\in\tilde{F}}$, such that 
$\sum_{p\in\tilde{F}}b_p\otimes\tilde{c}_p=\sum_{i=1}^n\sum_{p\in P_i}b_p\otimes c^{(i)}_p$. 
Therefore,
\begin{align}
\bigl\|(x\otimes1)E-\sum_{p\in\tilde{F}}b_p\otimes\tilde{c}_p\bigr\|
&\le\bigl\|([x-(b'_1bb''_1+b'_2bb''_2+\dots+b'_nbb''_n)]\otimes1)E\bigr\|
\notag \\
&\qquad+\sum_{i=1}^n
\bigl\|(b'_ibb''_i\otimes1)E-\sum_{p\in P_i}b_p\otimes c^{(i)}_p\bigr\|
\notag \\
&<\frac{\varepsilon}{3}\|E\|+n\cdot\frac{\varepsilon}{3n}<\varepsilon.
\notag
\end{align}
\end{proof}

\begin{rem}
The choice of the generators $\{b_p\}_{p\in P}$ is determined by $b$, but of course not unique. 
Even so, to avoid confusion, we will from now on fix the choice of the $\{b_p\}$ as the generators 
for $BbB$.  Without loss of generality, we may assume that each generator $b_p$ is a self-adjoint 
element.  In addition, since ${\mathfrak N}_{\nu}$ is dense in $B$, we may as well take each 
$b_p\in{\mathfrak N}_{\nu}$, or even $b_p\in{\mathcal T}_{\nu}$ (Tomita subalgebra).  This 
would in turn means that each $b_p$ is contained in the domains of the maps $\gamma_B$ 
and $\gamma_C^{-1}$.
\end{rem}

In the below is a lemma concerning the elements $\tilde{c}_p\in C$:

\begin{lem}\label{c_plem}
In Theorem~\ref{BbBgenerators}\,(2) above, as we approximate $(x\otimes1)E$ by 
the finite sums of the form $\sum b_p\otimes\tilde{c}_p$, it is possible to choose each 
$\tilde{c}_p$ to be contained in $\gamma_C^{-1}(BbB)$.
\end{lem}

\begin{proof}
Since ${\mathcal D}(\gamma_B)$ is dense in $B$, we may, without loss of generality, 
assume that $x\in {\mathcal D}(\gamma_B)\cap BbB$.  Let $\varepsilon>0$. 
As in (2) above, we can find a finite subset $\tilde{F}\subset P$ and the elements 
$\{c'_p\}_{p\in\tilde{F}}$ in $C$, such that $$\bigl\|(x\otimes1)E
-\sum_{p\in\tilde{F}}b_p\otimes c'_p\bigr\|<\frac{\varepsilon}2.$$
Since ${\mathcal D}(\gamma_C)$ is dense in $C$, we may assume here that each 
$c'_p\in{\mathcal D}(\gamma_C)$.

Apply $\gamma_B\otimes\gamma_C$.  Since $\gamma_B$ is anti-multiplicative, we then have:
$$
\bigl\|(\gamma_B\otimes\gamma_C)(E)(\gamma_B(x)\otimes1)-\sum_{p\in\tilde{F}}
\gamma_B(b_p)\otimes\gamma_C(c'_p)\bigr\|<\frac{\varepsilon}2.
$$
Recall from Proposition~\ref{sigmaE} that $(\gamma_B\otimes\gamma_C)(E)=\sigma E$. 
So applying the flip map, this becomes $\bigl\|E(1\otimes\gamma_B(x))-\sum_{p\in\tilde{F}}
\gamma_C(c'_p)\otimes\gamma_B(b_p)\bigr\|<\frac{\varepsilon}2$, or equivalently 
(by Proposition~\ref{sepidgamma}), we have: 
\begin{equation}\label{(c_plemeqn)}
\bigl\|E(x\otimes1)-\sum_{p\in\tilde{F}}\gamma_C(c'_p)\otimes\gamma_B(b_p)\bigr\|
<\frac{\varepsilon}2.
\end{equation}

Let $M=|\tilde{F}|\cdot\operatorname{max}\bigl(\|b_p\|:p\in\tilde{F}\bigr)$, and consider 
a (Hahn--Banach type) functional $\omega\in C^*$, such that $\omega\bigl(\gamma_B(b_p)\bigr)
=\frac1M$ and $\omega\bigl(\gamma_B(b_q)\bigr)=0$ if $q\in\tilde{F}$ with $q\ne p$. 
This is do-able because $\gamma_B$ being an injective anti-homomorphism, the linear 
independence of the $\{b_p\}$ implies the linear independence of the $\bigl\{\gamma_B(b_p)\bigr\}$. 
Note that $\|\omega\|=\frac1M$.  Applying $\operatorname{id}\otimes\omega$ to 
Equation~\eqref{(c_plemeqn)}, we obtain:
$$
\bigl\|(\operatorname{id}\otimes\omega)[E(x\otimes1)]-\gamma_C(c'_p)\bigr\|
<\frac{\varepsilon}{2M}.
$$
Note that $(\operatorname{id}\otimes\omega)[E(x\otimes1)]\in BbB$.  In other words, we must 
have each $\gamma_C(c'_p)$ to be within $\frac{\varepsilon}{2M}$-distance from $BbB$.  

While this in itself does not automatically mean each $\gamma_C(c'_p)$ must be contained 
in $BbB$, this allows us to choose instead $\tilde{c}_p\in\gamma_C^{-1}(BbB)$ such that 
$\|c'_p-\tilde{c}_p\|=\bigl\|\gamma_C(c'_p)-\gamma_C(\tilde{c}_p)\bigr\|<\frac{\varepsilon}{2M}$, 
for each $p\in\tilde{F}$.  We would then have: 
\begin{align}
\bigl\|(x\otimes1)E-\sum_{p\in\tilde{F}}b_p\otimes\tilde{c}_p\bigr\|
&\le\bigl\|(x\otimes1)E-\sum_{p\in\tilde{F}}b_p\otimes c'_p\bigr\|
+\sum_{p\in\tilde{F}}\|b_p\otimes(c'_p-\tilde{c}_p)\bigr\|
\notag \\
&<\frac{\varepsilon}2+|\tilde{F}|\cdot\operatorname{max}\bigl(\|b_p\|:p\in\tilde{F}\bigr)\cdot
\frac{\varepsilon}{2M}<\varepsilon.
\notag
\end{align}
\end{proof}

\subsection{$BbB$ is a (separable) liminal $C^*$-algebra}

Assume now that the fixed element $b\in B$ is self-adjoint.  Then $BbB$ is a (two-sided, 
closed) ideal in $B$, so itself a $C^*$-algebra.  Also by the characterization given 
in Theorem~\ref{BbBgenerators}, we see that $BbB$ is a separable $C^*$-algebra. 
Let us explore what else we can learn about the nature of this $C^*$-algebra $BbB$.

Since $\nu$ is a KMS weight on $B$, we can perform the GNS construction, and 
obtain $(\pi,{\mathcal H},\Lambda)$, where $\Lambda$ is the (injective) GNS map.  We may 
regard $B=\pi(B)\subseteq{\mathcal B}({\mathcal H})$.  Clearly, $\pi$ is also a representation 
for the $C^*$-algebra $BbB$, and we have: $BbB=\pi(BbB)\subseteq{\mathcal B}({\mathcal H})$.

Consider now the subspace ${\mathcal H}_b:=\overline{\operatorname{span}\bigl
(\Lambda(b_p):p\in P\bigr)}^{\|\ \|_{\mathcal H}}\,\subseteq{\mathcal H}$, where the 
$b_p\in{\mathcal T}_{\nu}$  are as in Theorem~\ref{BbBgenerators}.  Under the representation 
$\pi$, the $C^*$-algebra $BbB$ leaves the subspace ${\mathcal H}_b$ invariant.  In this way, 
we obtain a subrepresentation $\pi_b$ of $BbB$, acting on ${\mathcal H}_b$.  Meanwhile, 
by Gram--Schmidt process, without loss of generality, we may further assume 
that $\nu(b_qb_p)=\bigl\langle\Lambda(b_p),\Lambda(b_q)\bigr\rangle=0$ if $p,q\in P$ 
is such that $p\ne q$, while $\nu(b_pb_p)=1$, $\forall p\in P$.  (Recall that we are assuming 
each $b_p$ to be self-adjoint).

\begin{defn}
For any $p,q\in P$, denote by $\langle b_p,b_q\rangle_E$ the operator acting on ${\mathcal H}_b$ 
such that 
$$
\langle b_p,b_q\rangle_E\xi:=\bigl\langle\xi,\Lambda(b_q)\bigr\rangle\Lambda(b_p).
$$
Clearly, these are finite-rank (actually, rank one) operators contained in ${\mathcal B}
({\mathcal H}_b)$.
\end{defn}

\begin{prop}\label{propmu}
Any operator in $\pi_b(BbB)$ can be approximated by a linear combination 
of the operators $\langle b_p,b_q\rangle_E$, $p,q\in P$.
\end{prop}

\begin{proof}
Since $\pi_b$ arises from the GNS representation, we know that for $x\in BbB$ and 
$\Lambda(\tilde{b})\in{\mathcal H}_b$, we have: $\pi_b(x)\Lambda(\tilde{b})
=\Lambda(x\tilde{b})$.  

Let $\tilde{b}=b_m$, for $m\in P$.  As $\tilde{b}\in{\mathcal D}(\gamma_B)$, we can 
write $\tilde{b}=(\operatorname{id}\otimes\mu)\bigl(E(\tilde{b}\otimes1)\bigr)$. 
Then $x\tilde{b}=(\operatorname{id}\otimes\mu)\bigl((x\otimes1)E(\tilde{b}\otimes1)\bigr)
=(\operatorname{id}\otimes\mu)\bigl((x\otimes1)E(1\otimes\gamma_B(\tilde{b}))\bigr)$. 

Let $\varepsilon>0$ be arbitrary.  By Theorem~\ref{BbBgenerators}\,(2), we can find a finite 
subset $F\subset P$ and $\{\tilde{c}_p\}_{p\in F}$ in $C$, such that 
$\bigl\|(x\otimes1)E-\sum_{p\in F}b_p\otimes\tilde{c}_p\bigr\|<\frac{\varepsilon}{2\|\tilde{b}\|}$.  
By Lemma~\ref{c_plem}, we may assume that each $\tilde{c}_p\in\gamma_C^{-1}(BbB)$.

Observe that:
\begin{align}
\bigl\|x\tilde{b}-\sum_{p\in F}\mu(\tilde{c}_p\gamma_B(\tilde{b}))b_p\bigr\|
&=\bigl\|(\operatorname{id}\otimes\mu)\bigl([(x\otimes1)E-\sum_{p\in F}b_p\otimes\tilde{c}_p]
(1\otimes\gamma_B(\tilde{b}))\bigr)\bigr\|
\notag \\
&<\frac{\varepsilon}{2\|\tilde{b}\|}\cdot\bigl\|\gamma_B(\tilde{b})\bigr\|=\frac{\varepsilon}2.
\notag
\end{align}

We know $\mu=\nu\circ\gamma_C$.  So $\mu(\tilde{c}_p\gamma_B(\tilde{b}))
=\nu\bigl((\gamma_C\circ\gamma_B)(\tilde{b})\gamma_C(\tilde{c}_p)\bigr)$.  And, using 
the fact that $\gamma_C\circ\gamma_B=\sigma^{\nu}_i$ (see Proposition~\ref{modularauto}), 
we see that $\mu(\tilde{c}_p\gamma_B(\tilde{b}))=\nu\bigl(\gamma_C(\tilde{c}_p)\tilde{b}\bigr)$. 
As $\gamma_C(\tilde{c}_p)\in BbB$, we can approximate it arbitrarily closely by a finite 
linear combination of the $b_q$, for $q\in F^{(p)}\subset P$.  This can be done for each 
$p\in F$. Therefore, it is evident that we can find a finite linear combination of the 
expressions $\nu(b_q\tilde{b})b_p=\bigl\langle\Lambda(\tilde{b}),\Lambda(b_q)
\bigr\rangle\,b_p$, that is within $\frac{\varepsilon}2$-distance from $\sum_{p\in F}
\mu(\tilde{c}_p\gamma_B(\tilde{b}))b_p$.  

It follows that $x\tilde{b}$ can be approximated within $\varepsilon$-distance by 
a finite linear combination of the $\bigl\langle\Lambda(\tilde{b}),\Lambda(b_q)\bigr
\rangle\,b_p$.  Applying the GNS-map $\Lambda$, we see that $\pi_b(x)\Lambda(\tilde{b})$ 
can be approximated arbitrarily closely by finite linear combinations of the expressions 
$\bigl\langle\Lambda(\tilde{b}),\Lambda(b_q)\bigr\rangle\,\Lambda(b_p)
=\langle b_p,b_q\rangle_E\,\Lambda(\tilde{b})$.

All this would hold for each of the basis element $\Lambda(b_m)$, $m\in P$.
In this way, we prove that any $\pi_b(x)$, $x\in BbB$, can be approximated by 
a linear combination of the operators $\langle b_p,b_q\rangle_E$, $p,q\in P$.
\end{proof}

Write $V=\operatorname{span}\bigl(\langle b_p,b_q\rangle_E:p,q\in\Gamma\bigr)\,\bigl(\subseteq
{\mathcal B}({\mathcal H}_b)\bigr)$.  Then: 

\begin{prop}\label{Vcompactops}
For the operators $\langle b_p,b_q\rangle_E,\langle b_{p'},b_{q'}\rangle_E\in V$, we have:
\begin{enumerate}
  \item $\langle b_p,b_q\rangle_E\langle b_{p'},b_{q'}\rangle_E=\delta_{p',q}\langle b_p,b_{q'}\rangle_E$
  \item $\langle b_p,b_q\rangle_E^*=\langle b_q,b_p\rangle_E$.
\end{enumerate}
As a result, we conclude that $\overline{V}^{\| \|_{\operatorname{op}}}\cong{\mathcal B}_0
({\mathcal H}_b)$, the algebra of compact operators on ${\mathcal H}_b$.
\end{prop}

\begin{proof}
These are rather standard and straightforward.

(1). Let $\xi\in{\mathcal H}_b$ be arbitrary.  we have:
\begin{align}
\langle b_p,b_q\rangle_E\langle b_{p'},b_{q'}\rangle_E\xi
&=\bigl\langle\xi,\Lambda(b_{q'})\bigr\rangle\langle b_p,b_q\rangle_E\Lambda(b_{p'}) 
\notag \\
&=\bigl\langle\xi,\Lambda(b_{q'})\bigr\rangle\bigl\langle\Lambda(b_{p'}),\Lambda(b_{q})\bigr\rangle
\Lambda(b_p).  \notag \\
&=\nu(b_q^*b_{p'})\langle b_p,b_{q'}\rangle_E\xi=\delta_{p',q}\langle b_p,b_{q'}\rangle_E\xi.
\notag
\end{align}

(2). Let $\xi,\zeta\in{\mathcal H}_b$.  Then
\begin{align}
\bigl\langle\langle b_p,b_q\rangle_E\xi,\zeta\big\rangle
&=\bigl\langle\langle\xi,\Lambda(b_q)\rangle\Lambda(b_p),\zeta\bigr\rangle
=\bigl\langle\xi,\Lambda(b_q)\bigr\rangle\bigl\langle\Lambda(b_p),\zeta\bigr\rangle
\notag \\
&=\bigl\langle\xi,\overline{\langle\Lambda(b_p),\zeta\rangle}\Lambda(b_q)\bigr\rangle
=\bigl\langle\xi,\langle\zeta,\Lambda(b_p)\rangle\Lambda(b_q)\bigr\rangle
\notag \\
&=\bigl\langle\xi,\langle b_q,b_p\rangle_E\zeta\bigr\rangle.
\notag
\end{align}
It follows that $\langle b_p,b_q\rangle_E^*=\langle b_q,b_p\rangle_E$.

(3). The results of (1), (2) show that $V$ is a ${}^*$-algebra contained in ${\mathcal B}
({\mathcal H}_b)$, and the generators behave exactly like the matrix units (see, for instance, 
Lemma~\ref{leme_ij}).  Therefore, the operator norm closure of $V$ coincides with 
${\mathcal B}_0({\mathcal H}_b)$.
\end{proof}

The following results are immediate:

\begin{theorem}\label{BbBthm}
Let $b\in B$ be a fixed self-adjoint element.  Then we have:
\begin{enumerate}
  \item $BbB=\pi_b(BbB)\subseteq{\mathcal B}_0({\mathcal H}_b)$.
  \item The $C^*$-algebra $BbB$ is liminal.
\end{enumerate}
\end{theorem}

\begin{proof}
(1). The result of Proposition~\ref{propmu} is that $\pi_b(BbB)$ is contained in the closure of $V$. 
Therefore, by Proposition~\ref{Vcompactops}, the result follows.

(2). Since $BbB$ is a $C^*$-algebra contained in the algebra of compact operators, 
it is {\em liminal\/}, in the sense of \cite{Dixbook}.
\end{proof}

\subsection{$B$ is a postliminal $C^*$-algebra}

By definition, a $C^*$-algebra $A$ is said to be {\em postliminal\/}, if every non-zero quotient 
$C^*$-algebra of $A$ possesses a non-zero liminal closed two-sided ideal (see \cite{Dixbook}, 
\cite{blackbook}).

In our case, suppose $J$ is any nontrivial proper (closed, two-sided) ideal of $B$, and 
let $B/J$ be its corresponding quotient $C^*$-algebra.  Since $J\ne B$, we can consider 
a self-adjoint element $b\in B$ with $b\not\in J$.  Consider the ideal $BbB$ of $B$.  Then 
$BbB+J$ is an ideal of $B$ such that $(BbB+J)/J\cong BbB/(BbB\cap J)$.  Since the latter is 
a quotient of a liminal $C^*$-algebra $BbB$, it is also liminal.  In this way, we have shown that 
the quotient $C^*$-algebra $B/J$ contains an ideal $(BbB+J)/J$, which is isomorphic to a 
liminal $C^*$-algebra.  By definition above, we can see that $B$ is {\em postliminal\/}:

\begin{theorem}
Let $(E,B,\nu)$ be a separability triple.  Then $B$ is a postliminal $C^*$-algebra.
\end{theorem}

\begin{rem}
For  separable $C^*$-algebras, it is well-known that being postliminal is equivalent to being 
{\em type~I\/}.
The result that $B$ is postliminal (or type I) seems compatible with the result in the 
purely algebraic case \cite{VDsepid}, where it was shown that the base algebra $B$ has 
to be a direct sum of matrix algebras.
\end{rem}

Finally, here is the result when $B$ is commutative:

\begin{prop}
If $B$ is commutative, its spectrum is totally disconnected.
\end{prop}

\begin{proof}
If $B$ is commutative, any of its ideals is also commutative.  In particular, each building-block 
ideal $BbB$ (for $b$ self-adjoint), being a commutative subalgebra contained in the algebra 
of compact operators, would be isomorphic to $\mathbb{C}$, spanned by a single projection. 
Our $B$ would be spanned by its projections. If we write $B\cong C(X)$, where $X$ is the 
spectrum of $B$, this means that $X$ is totally disconnected.
\end{proof}

\bigskip\bigskip

%\bibliography{refartikel}

%\bibliographystyle{amsplain}

\providecommand{\bysame}{\leavevmode\hbox to3em{\hrulefill}\thinspace}
\providecommand{\MR}{\relax\ifhmode\unskip\space\fi MR }
% \MRhref is called by the amsart/book/proc definition of \MR.
\providecommand{\MRhref}[2]{%
  \href{http://www.ams.org/mathscinet-getitem?mr=#1}{#2}
}
\providecommand{\href}[2]{#2}

\end{document}